\documentclass[11pt,reqno,a4paper]{amsart}
\pdfoutput=1
\usepackage{latexsym,amsthm,amssymb,amsmath,amscd,mathtools,hyperref,xpatch}
\mathtoolsset{showonlyrefs}
\usepackage{tikz-cd}
\tikzcdset{
  cells={font=\everymath\expandafter{\the\everymath\displaystyle}},
}
\usetikzlibrary{arrows.meta,positioning,automata,calc}
\usetikzlibrary{matrix,arrows,decorations.pathmorphing}
\usepackage{bm}
\usepackage{fancyhdr}
\usepackage{mathtools}
\usepackage{mathrsfs} 
\usepackage{tcolorbox}
\usepackage{bbm}

\usepackage[normalem]{ulem}

\usepackage[all]{xy}

\usepackage[top=25truemm, bottom=20truemm,right=20truemm,left=20truemm]{geometry}

\usepackage{graphicx}

\usepackage[shortlabels]{enumitem}
\setlist[enumerate]{nosep}

\usepackage{xcolor}
\newcommand\redsout{\bgroup\markoverwith{\textcolor{red}{\rule[0.5ex]{2pt}{0.8pt}}}\ULon}

\newtheorem{theorem}{Theorem}[section]

\newtheorem*{conjecture*}{Conjecture}
\newtheorem*{theorem*}{Theorem}
\newtheorem{lemma}[theorem]{Lemma}
\newtheorem*{lemma*}{Lemma}
\newtheorem{corollary}[theorem]{Corollary}
\newtheorem*{corollary*}{Corollary}
\newtheorem{proposition}[theorem]{Proposition}
\newtheorem{remark}[theorem]{Remark}
\newtheorem{definition}[theorem]{Definition}
\newtheorem*{definition*}{Definition}
\newtheorem*{definitions*}{Definitions}

\newtheorem*{question*}{Question}

\newtheorem*{questions*}{Questions}

\newtheorem{thm}{Theorem}[section]

\newtheorem{cor}{Corollary}[section]

\newcommand{\cG}{\mathcal G}
\newcommand{\cH}{\mathcal H}

\newcommand{\cR}{\mathcal R}

\newcommand{\cZ}{\mathcal Z}

\def\Cz{\mathbb{C}}

\def\Nz{\mathbb{N}}

\def\Pz{\mathbb{P}}

\def\Rz{\mathbb{R}}

\def\Zz{\mathbb{Z}}

\def\1z{\mathbb{1}}
\newcommand{\fS}{\mathfrak S}

\newcommand{\sgn}{\textup{sgn}}

\def\SEMI{\mbox{$\times\kern-2pt\vrule height5pt width.6pt \kern3pt $}}

\newcommand{\Hom}{{\rm Hom}\,}

\newcommand{\Aut}{{\rm Aut}\,}

\newcommand{\Coker}{{\rm Coker}\,}
\newcommand{\id}{{\rm id}}

\newcommand{\res}{{\rm res}}

\newcommand{\mfr}{\mathfrak{r}}
\newcommand{\mfs}{\mathfrak{s}}
\newcommand{\G}{\mathcal{G}}

\newcommand{\m}{\mathfrak{m}}

\renewcommand{\r}{\mathfrak{r}}
\newcommand{\s}{\mathfrak{s}}

\newcommand{\tail}{\textup{tail}}
\newcommand{\prim}{\textup{prim}}

\newcommand{\acts}{\curvearrowright} 

\renewcommand{\Im}{\textup{Im}}
\newcommand{\eps}{\varepsilon}

\newcommand{\ol}{\overline}

\newcommand{\Ker}{\textup{Ker}}

\newcommand{\Ad}{\textup{Ad}}

\newcommand{\Curr}{\textup{Curr}}

\newcommand{\COE}{\textup{COE}}

\newcommand{\D}{\bm{D}}
\newcommand{\F}{\bm{F}}

\renewcommand{\H}{\textup{H}}
\newcommand{\E}{\textup{E}}

\newcommand{\Homeo}{\textup{Homeo}}
\renewcommand{\Aut}{\textup{Aut}}
\newcommand{\Out}{\textup{Out}}
\newcommand{\Inn}{\textup{Inn}}
\newcommand{\Stab}{\textup{Stab}}

\newcommand{\cycle}{{\rm cycle}}

\begin{document}

\title[Boundary actions of outer automorphism groups of Thompson-like groups]{Boundary actions of outer automorphism groups of Thompson-like groups}

\thispagestyle{fancy}

\author{Chris Bruce}
\address[Chris Bruce]{School of Mathematics, Statistics and Physics, Herschel Building, Newcastle University, Newcastle upon Tyne, NE1 7RU, United Kingdom}
\email[Bruce]{chris.bruce@newcastle.ac.uk}

\author{Xin Li}
\address[Xin Li]{School of Mathematics and Statistics, University of Glasgow, University Place, Glasgow G12 8QQ, United Kingdom}
\email[Li]{Xin.Li@glasgow.ac.uk}

\author{Takuya Takeishi}
\address[Takuya Takeishi]{Faculty of Arts and Sciences, Kyoto Institute of Technology, Matsugasaki, Sakyo-ku, Kyoto, Japan}
\email[Takeishi]{takeishi@kit.ac.jp}

\date{\today}

\begin{abstract}
    For every Cuntz--Krieger groupoid, we show that there is a topologically free boundary action of the outer automorphism group of its topological full group on the Hilbert cube. In particular, these outer automorphism groups, including the outer automorphism groups of all Higman--Thompson groups, are C*-simple. 
\end{abstract}

\maketitle

\setlength{\parindent}{0cm} \setlength{\parskip}{0.5cm}

\section{Introduction}

\subsection*{Higman--Thompson groups}
In 1965, Thompson introduced three groups $F<T<V$. The groups $T$ and $V$ were the first examples of infinite finitely presented simple groups, see \cite{CFP96}. Higman generalised Thompson's construction of $V$ to define the first infinite family of finitely presented infinite simple groups \cite{Higman74}. Namely, Higman defined a family of groups $\{V_{n,r}\}$ such that $V_{2,1}=V$ and the group $V_{n,r}$ is finitely presented, simple when $n$ is even, and has a simple finitely presented commutator subgroup of index $2$ when $n$ is odd. These groups are called the \emph{Higman--Thompson groups}, and they have been studied intensely in the 50-year period since their conception. 

Thirty years ago, Brin initiated the study of automorphism groups of Thompson's groups, where he described the groups $\Out(F)$ and $\Out(T)$ \cite{Brin96}. He left open the natural problem of analysing $\Out(V)$ or $\Out(V_{n,r})$. It was only recently proven by Bleak, Cameron, Maissel, Navas, and Olukoya that $\Out(V_{n,r})$ is an infinite group, using a description of its elements as transducers \cite{BCMNO24}. Using the framework from \cite{BCMNO24}, it was then proven that $\Out(V_{n,r})$ contains a copy of the automorphism groups of the full one-sided $n$-shift and copies of the full two-sided $m$-shifts for $m\geq 2$ \cite{BBCO21,BCO21,Olu24} and that $\Out(V_{n,1})$ contains a copy of Thompson's group $F$ for $n\geq 3$ \cite{Olu22}. Since automorphism groups of the two-sided full shifts contain non-abelian free subgroups and all finite groups, this implies, e.g., that $\Out(V_{n,r})$ contains all finite groups and has non-abelian free subgroups. In particular, $\Out(V_{n,r})$ is non-amenable. It was proven in \cite{BBCO21} that $\Out(V_{n,r})$ has a trivial centre and unsolvable order problem. Aside from these results, little is known about the structure of $\Out(V_{n,r})$, despite the long-standing interest in the Higman--Thompson groups.

\subsection*{Topological full groups}
\emph{Topological full groups} attached to Cantor minimal systems were introduced by Giordano, Putnam, and Skau \cite{GPS99}. This construction has led to several breakthroughs in group theory. Matui proved that the topological full groups from Cantor minimal systems have simple commutator subgroups \cite{Mat06}, and Juschenko and Monod proved that they are amenable, leading to the first examples of infinite finitely generated simple amenable groups \cite{JM13}. Matte Bon used topological full groups for Cantor minimal systems to construct the first examples of infinite, finitely generated, simple groups with the Liouville property \cite{MBon14}.
Matui then generalised the construction of topological full groups to \'{e}tale groupoids over the Cantor set: each such groupoid gives rise to a discrete group of homeomorphisms of the Cantor set that consists of the homeomorphisms that locally look like the action of the groupoid on its unit space \cite{Mat12}. Topological full groups play a crucial role in Nekrashevych's construction of the first examples of simple groups of intermediate growth \cite{Nek18}, and the Nekrashevych-R\"{o}ver groups from \cite{Nek04}---which are topological full groups associated with the groupoids from self-similar group actions \cite{Nek09}---were proven by Skipper, Witzel, and Zaremsky to be the first examples of simple groups that are of type $F_n$ but not of type $F_{n-1}$ \cite{SWZ19}.

Some of the most important \'{e}tale groupoids are the \emph{Cuntz--Krieger groupoids}, which arise as groupoid models for the Cuntz--Krieger C*-algebras attached to non-trivial irreducible shifts of finite type \cite{Cun81,CK80}. 
Every shift of finite type can be modelled as the one-sided edge shift $(X_E,\sigma_E)$ of a finite directed graph $E$ (see \cite{LM21}), and the Cuntz--Krieger groupoid $\cG_E$ is the groupoid associated with the shift of finite type $(X_E,\sigma_E)$ (see \cite[Chapter~5]{MatBook} for details). 
Nekrashevych showed that the Higman--Thompson group $V_{n,1}$ (e.g., Thompsons group $V=V_{2,1}$) could be realsied as the topological full group of the Cuntz--Krieger groupoid for the full $n$-shift, and Matui extended this to show that every Higman--Thompson group appears as the topological full group of a Cuntz--Krieger groupoids (see \cite{Nek04} and \cite[Remark 6.3 and Section 6.7.1]{Mat15}). Matui then initiated a systematic study of the topological full groups of Cuntz--Krieger groupoids \cite{Mat15}, where he proved that these groups share many of the important properties satisfied by the Higman--Thompson groups. In fact, Matui proved that for every Cuntz--Krieger groupoid $\cG_E$, its topological full group $\F(\cG_E)$ is of type $F_\infty$ (in particular, it is finitely presented) and that the commutator subgroup $\D(\cG_E)$ of $\F(\cG_E)$ is simple and finitely generated. Matui further showed that $\F(\cG_E)$ is not amenable, but has the Haagerup property, and he gave an explicit description of the abelianisation of $\F(\cG_E)$ in terms of the adjacency matrix of the graph $E$. In light of these results, the topological full groups $\F(\cG_E)$ attached to Cuntz--Kriger groupoids are natural generalisations of the Higman--Thompson groups, and we shall refer to them as \emph{Thompson-like groups}.

In many cases, the topological full group of an \'{e}tale groupoid completely remembers the groupoid \cite{Mat15,Rubin89}, so that one can go back and forth between the groupoid and the topological full group. This allows one to use powerful groupoid techniques and ideas from topological dynamics and algebraic topology to study groups in new ways. For instance, the second-named author in \cite{Li25} recently established a deep connection between the group homology of topological full groups and groupoid homology. This led to a resolution of Matui's AH conjecture from \cite{Mat16}, to novel homology calculations for groups, and to a new proof that Thompson's group $V$ is integrally acyclic, recovering the main result from \cite{SW19}. Further, the topological full groups attached to Cuntz--Krieger groupoids were used in the recent proof of the Boone--Higman conjecture for hyperbolic groups \cite{BBMZ25}. In this paper, we use groupoid techniques to construct natural boundary actions of the outer automorphism groups $\Out(\F(\cG_E))$.

\subsection*{C*-simplicity}
A group is said to be \emph{C*-simple} if its reduced C*-algebra is a simple C*-algebra. C*-simplicity of discrete groups has been of great interest to the group theory and operator algebra communities since the work of Powers, who proved that non-abelian free groups are C*-simple \cite{Powers75}. See \cite{delaHarpe07} for a survey of the earlier works on C*-simplicity.
In a breakthrough, Kalantar and Kennedy proved that a discrete group $G$ is C*-simple if and only if it acts (topologically) freely on its Furstenberg boundary if and only if it admits a topologically free boundary action \cite{KK17}. Breuillard, Kalantar, Kennedy, and Ozawa then gave a new proof of this result \cite{BKKO17}, leading to a representation-theoretic characterisation of C*-simplicity: a discrete group $G$ is C*-simple if and only if for every amenable subgroup of $G$, the associated quasi-regular representation is weakly equivalent to the regular representation. They also established C*-simplicity for several large classes of groups and proved that the reduced C*-algebra of $G$ has a unique trace if and only if the amenable radical of $G$ vanishes. In particular, this showed that C*-simplicity implies the unique trace property. For many classes of groups, vanishing of the amenable radical is equivalent to C*-simplicity, see, e.g., \cite[Theorems 1.6 \& 1.7]{BKKO17}. However, Le Boudec proved that vanishing of the amenable radical is not in general sufficient for C*-simplicity \cite{Boudec17}.

C*-simplicity has several other striking consequences and equivalences, in addition to the ones mentioned above. Kennedy and Haagerup \cite{Haag15,Ken20} independently proved that C*-simplicity is equivalent to the \emph{Powers' averaging property} from \cite{Powers75}, and Kennedy further showed that a discrete group $G$ is C*-simple if and only if it has no non-trivial amenable uniformly recurrent subgroups (in the sense of Glasner and Weiss \cite{GW15}). This reinforces de la Harpe's remark in \cite{delaHarpe07} that C*-simplicity is viewed as "an extreme case of non-amenability". Hartman and Kalantar showed that C*-simplicity of $G$ is equivalent to the action of $G$ on the space of amenable subgroups of $G$ being uniquely $\mu$-stationary for some probability measure $\mu$ on $G$ \cite[Corollary~5.7]{HK23}.

Kennedy's characterisation of C*-simplicity, combined with the main result in \cite{LBMB18}, gives a powerful tool for proving that certain classes of groups are C*-simple: if $G$ admits a micro-supported action on a compact Hausdorff space for which all the rigid stabiliser subgroups are non-amenable, then $G$ must be C*-simple. Note that such actions are necessarily very far from being topologically free. In fact, a typical example is the action of a topological full group of a purely infinite groupoid on the unit space of the groupoid, and the results from \cite{LBMB18} lead to C*-simplicity for many topological full groups, see \cite{BS19}. Aside from topological full groups, the largest classes of C*-simple groups are the linear groups or the acylindrically hyperbolic groups with trivial amenable radicals, see \cite[Theorem~1.6]{BKKO17} and \cite[Theorem~8.14]{DGO17}, respectively. There are C*-simple groups that do not fit into these classes, see, e.g., \cite{KT-D23, OO14}, and the outer automorphisms of Thompson-like groups also do not fit into one of the usual classes of C*-simple groups.

In general, it can be very difficult to determine whether a given group or class of groups is C*-simple, and even when it is known that a given group is C*-simple, it can be hard to find an explicit topologically free boundary action on a second countable space. These difficulties are witnessed in particular for Thompson's groups: C*-simplicity of Thompson's group $T$ is equivalent to non-amenability of $F$ \cite{HO17,LBMB18}, which is a long-standing open problem. It is proven in \cite[Theorem~4.5]{LBMB18} that Thompson's group $V$ is C*-simple, but the proof uses the characterisation of C*-simplicity from \cite[Theorem~1.2]{Ken20}, which does not provide any explicit boundary action. 
C*-simplicity of the Thompson-like groups $\F(\cG_E)$ is also known only from the general results in \cite{LBMB18} (see \cite[Corollary~5.3]{BS19}), and consequently, no explicit topologically free boundary action of $\F(\cG_E)$ on a second countable space is known. Surprisingly, the situation for the outer automorphism groups $\Out(\F(\cG_E))$ is quite different.

\subsection*{The main result}

In this article, we prove the following theorem.
\begin{thm}
\label{thm:main} 
    For every Cuntz--Krieger groupoid $\cG_E$, the outer automorphism group 
    $\Out(\F(\cG_E))$ admits a topologically free boundary action on the Hilbert cube.
    In particular, $\Out(\F(\cG_E))$ is C*-simple.
\end{thm}
Especially, $\Out(V_{n,r})$ has such a boundary action for all Higman--Thompson groups $V_{n,r}$. 
Here, the Hilbert cube is realised as $\Pz M(X_E,\sigma_E)_+$, the projectivisation of the space of positive, shift-invariant Radon measures on the shift of finite type $X_E$.

Since $\F(\cG_E)$ is C*-simple, $\Aut(\F(\cG_E))$ is automatically C*-simple by \cite[Proposition~2.23]{LBMB18}. However, for a group $G$, there is no general relationship between C*-simplicity of $G$ and C*-simplicity of $\Out(G)$. For example, torsion-free hyperbolic groups are C*-simple \cite{delaHarpe88,delaHarpe07}, and every finite group arises as the outer automorphism group of a torsion-free hyperbolic group, see \cite{CIOS25,Paulin88}.

By \cite{Ken20}, $\Out(\F(\cG_E))$ has no non-trivial amenable uniformly recurrent subgroups; in particular, $\Out(\F(\cG_E))$ has trivial amenable radical, which in turn implies that $\Out(\F(\cG_E))$ has trivial center, recovering \cite[Corollary~1.5]{BBCO21} for the Higman--Thompson groups.

Let us mention a connection between our work and C*-simplicity of outer automorphism groups of free groups, which was proven by Bridson and de la Harpe using the action on outer space \cite{BdlH04}. The idea of our proof comes from Uyanik's work \cite{Uyanik14}. 
Let $F_N$ be a free group of finite rank $N\geq 2$, and denote by $\partial F_N$ its Gromov boundary. The transformation groupoid $F_N\ltimes\partial F_N$ is a Cuntz--Krieger groupoid \cite{Mat15}, and it is not difficult to see that the restriction of the action $\Out(F_N\ltimes\partial F_N)\acts \Pz M(\partial F_N,\sigma)_+$ to the subgroup $\Out(F_N)$ is conjugate to the action of $\Out(F_N)$ on the space $\Pz\Curr(F_N)$ of projectivised currents of $F_N$ (cf. \cite{Kap05,Kap06}). The action of $\Out(F_N)$ on the minimal set in $\Pz\Curr(F_N)$ is a topologically free boundary action (cf. \cite{KL07,Uyanik14}).
Our proof is inspired by this case, but is almost entirely different at the technical level.

\subsection*{Outline of the proof}

In Section~\ref{sec:outer}, we introduce \emph{outer automorphism groups of groupoids} $\Out(\cG_E)$. The group $\Out(\cG_E)$ is canonically isomorphic to $\Out(\F(\cG_E))$ (Proposition~\ref{prop:aut}), which is essentially due to Matui \cite{Mat15}. The automorphisms of the groupoid $\cG_E$ are described via \emph{continuous orbit equivalences (COE)} of the one-sided shift of finite type (SFT) $(X_E,\sigma_E)$. There is a series of works of Matsumoto and Matsumoto--Matui \cite{Matsu10, Matsu15, MM14, MM16} on COEs of SFTs, and our work is based on their research. Their works are nicely summarised in \cite{MatBook}, so we cite this book instead of the original articles. 
The action of $\Out(\F(\cG_E))$ on $\Pz M(X_E,\sigma_E)_+$ is induced from the action of $\Out(\F(\cG_E))$ on the cohomology group $\H^1(\cG_E)\cong \H^E$, where $\H^E$ is the dynamical cohomology of the SFT. 
In Section~\ref{ssec:PM+}, we develop tools to analyse the action  $\Out(\F(\cG_E)) \acts \Pz M(X_E,\sigma_E)_+$. 
One important idea is to reduce the analysis of the classes of measures to the classes of \emph{periodic measures} $\eta_{[p]}$, where $p$ is a primitive cycle of the graph. The classes of periodic measureas are dense in $\Pz M(X_E,\sigma_E)_+$ (Proposition~\ref{prop:periodicmeasures}). This is essentially due to Sigmund \cite{Sig74}, which is a variant of Birkhoff's ergodic theorem. 
We then show that the action on the periodic measures is the same as actions on the classes of primitive cycles (Proposition~\ref{prop:varphieta[p]}), and then we observe that periodic measures can be analysed via combinatoric data. 
In order to construct sufficiently many COEs, we introduce \emph{marker COEs} in Section~\ref{ssec:marker}, which is inspired by marker automorphisms of shift spaces \cite{Ash90,BFK90, BLR88}. 
When we want to obtain marker COEs, we need to verify the \emph{overlap conditions}, which guarantee bijectivity of marker COEs. The overlap conditions are not easy to verify in an abstract setting, so there are several technical combinatorial arguments to check the overlap conditions. 
In Section~\ref{ssec:good}, we establish a way to analyse marker COEs via a combinatorial method by introducing \emph{good representatives}. 

In Section~\ref{sec:faithful-transitive} and Section~\ref{ssec:proximal}, we prove that the action $\Out(\F(\cG_E)) \acts \Pz M(X_E,\sigma_E)_+$ is a topologically free boundary action. The proof consists of 3 steps: (1) faithfulness of $\Out(\cG_E) \acts [E^*_\prim]$, (2) transitivity of $\Out(\cG_E) \acts [E^*_\prim]$, and (3) strong proximality of $\Out(\F(\cG_E)) \acts \Pz M(X_E,\sigma_E)_+$. Here, $[E^*_\prim]$ denotes the set of all cyclic-permutation classes of primitive cycles of $E$. 
The result on faithfulness is a generalisation of \cite[Theorem~1.4]{BBCO21}, and its proof 
is based on a theorem by Boyle--Krieger \cite[Theorem~2.5]{BK87} and Matsumoto's \emph{strong COEs} \cite{Matsu15}. Note that the result of Boyle--Krieger is based on the classical Curtis–-Hedlund–-Lyndon theorem \cite[Theorem 6.2.9]{LM21}, which is a classification of automorphisms of two-sided SFTs. The proof of transitivity is purely combinatorial (it is a kind of word problem). The proof consists of several arguments to verify overlap conditions or to examine the primitivity of given words. 
The most complicated part of this article is the proof of strong proximality. 
The basic idea of the key lemma (Lemma~\ref{lem:converge-final}) is as follows: if $f_n,g_n \colon X \to X$ are ``very good" sequences of homeomorphisms on a compact space $X$ and if we have $g_n(x) \to y$ and $f_n(y) \to z$ for a given $x \in X$, then we can expect that $f_ng_n(x) \to z$. Usually, the ``very good" condition should be equicontinuity. However, the marker COEs we use in Section~\ref{ssec:proximal} are not equicontinuous (Remark~\ref{rem:discontinuous}). Instead, we establish several technical lemmas, which amount to equicontinuity in a partial sense.  

Using the classification theorem of Cuntz--Krieger groupoids \cite{MM14}, 
every Cuntz--Krieger groupoid is isomorphic to a Cuntz--Krieger groupoid associated to a primitive graph. This is a well-known fact for operator algebraists, and we use a stronger version of this fact (Proposition~\ref{prop:EKTW}) at the final step of the proof of strong proximality. However, for Proposition~\ref{prop:periodicmeasures} and Proposition~\ref{prop:preservesorbits->inner}, we give an elementary proof using the \emph{periodic decomposition} of SFTs (Section~\ref{ssec:SFT}). In other words, the proofs of those results can be shortened by using the classification. 
In addition, we give an outline of alternative proofs of strong proximality for $\Out(V_2)$ and $\Out(V_{n,r})$ with $r \geq 2$ without using the classification theorem for the readers' convenience who are not familiar with operator algebras  (Section~\ref{ssec:remark}).

\subsection*{Commutator subgroups of Thompson--like groups}
A consequence of Theorem~\ref{thm:main} is the following description of the amenable radical and characterisation of C*-simplicity for the outer automorphism groups of the simple commutator subgroups $\D(\cG_E)$ arising from Cuntz--Krieger groupoids.
\begin{cor}
\label{cor}
For every Cuntz--Krieger groupoid $\cG_E$, the amenable radical of $\Out(\D(\cG_E))$ is isomorphic to 
\begin{equation}
\F(\cG_E)/\D(\cG_E) \cong(\Coker(I-A_E^t)\otimes\Zz/2\Zz)\oplus \Ker(I-A_E^t), 
\end{equation} 
where $A_E$ is the adjaceny matrix of $E$. 
Moreover, $\Out(\D(\cG_E))$ is C*-simple if and only if it has trivial amenable radical.   
\end{cor}
We can equivalently say that $\Out(\D(\cG_E))$ is C*-simple if and only if $\H_0(\cG_E)$ is finite and without $2$-torsion. In particular, $\Out([V_{n,r},V_{n,r}])$ is C*-simple if and only if $n$ is even. 
In addition, Corollary~\ref{cor} shows that the amenable radical of $\Out(\D(\cG_E))$ can be infinite, and this behaviour is different from acylindrically hyperbolic groups, which always have finite amenable radicals. 

\subsection*{Acknowledgements}
C. Bruce would like to thank Kevin Aguyar Brix for many helpful and inspiring conversations about SFTs and outer automorphism groups of groupoids.
This work began at a Research-in-Groups Programme at the ICMS, and we gratefully acknowledge their support.
X. Li has received funding from the European Research Council (ERC) under the European Union’s Horizon 2020 research and innovation programme (grant agreement No. 817597).
T. Takeishi is supported by JSPS KAKENHI grant number 24K06780.

\section{Preliminaries}

\subsection{Boundary actions and C*-simplicity}
Let $G$ be a discrete group and $X$ a compact Hausdorff space equipped with an action of $G$ by homeomorphisms. Recall that the action $G\acts X$ is \emph{minimal} if the orbit $Gx$ is dense in $X$ for every $x\in X$ and is \emph{topologically free} if $\{x\in X : \Stab_G(x)=\{e\}\}$ is dense in $X$, where $\Stab_G(x)$ is the stabiliser subgroup of $x$.
The action $G\acts X$ is said to be \emph{strongly proximal} if for every probability measure $\nu\in P(X)$, the weak$^*$ closure of the orbit $G\nu$ in $P(X)$ contains a point-mass measure $\delta_x$ for some $x\in X$. Here, $G$ acts on the compact space $P(X)$ of Radon probability measures by push-forward. 
The action $G\acts X$ is said to be a \emph{boundary action} if it is minimal and strong proximal. 
The \emph{amenable radical} of $G$ is the largest amenable normal subgroup of $G$, see \cite[Section 4, Lemma 1]{Day57}.

The group $G$ is said to be \emph{C*-simple} if its reduced group C*-algebra is a simple C*-algebra. 
A remarkable characterisation of C*-simplicity was proven by Kalantar and Kennedy:
\begin{theorem}[{\cite[Theorem~6.2]{KK17}}] \label{thm:KK}
A discrete group $G$ is C*-simple if and only if it admits a topologically free boundary action on a compact Hausdorff space.
\end{theorem}

\subsection{Finite graphs}
Throughout this paper, graphs are assumed to be directed, and $E=(E^0,E^1)$ will denote a finite graph with range and source maps $r$ and $s$, respectively. Let 
\[ E^n := \{ (e_0,\dots,e_{n-1}) \in (E^1)^n \colon r(e_i)=s(e_{i+1})\mbox{ for } i=0,\dots,n-2\}\]
and let $E^*=\bigcup_{n\geq 1} E^n$. 

An element $w \in E^n$ is called a \emph{path}. For path $w=(e_0,\dots,e_{n-1}) \in E^n$, let $w_{[i]} = e_i$ and $|w|=n$ (called the \emph{length} of $w$). Paths are written like $w=w_{[0]}w_{[1]}\dots w_{[n-1]}$. For $0 \leq k,l\leq |w|$, let $w_{[k,l]} = w_{[k]}\dots w_{[l]}$. 

\emph{The empty path} is denoted by $o$. In our convention, we do not consider $o$ as an element of $E^*$. We consider $o$ composable to all edges. 

\begin{definition}
    For paths $v, w \in E^*$, \emph{a $w$-segment in $v$} is an interval of integers $[i,j] \subset \Zz$ such that $v_{[i,j]} = w$. 
    For $S \subseteq E^*$, \emph{an $S$-segment in $v$} is an interval of integers $[i,j] \subset \Zz$ such that $w_{[i,j]}$ is equal to some $s \in S$. 
     For $v, w \in E^*$, $v$ is a \emph{prefix} of $w$ if $v=w_{[0,k]}$ for some $k$, and $v$ is a \emph{suffix} of $w$ if $v=w_{[k,|w|-1]}$ for some $k$. 
\end{definition}

\begin{definition}
    A path $p \in E^*$ is called a \emph{cycle} if $s(p)=r(p)$. A \emph{loop} is a cycle of length $1$. A cycle $p$ is \emph{prime} if there is no edge in $p$ whose source is $s(p)$ except for the first edge. A cycle $p$ is \emph{primitive} if $p=q^n$ for some cycle $q \in E^*$ and $n \geq 1$ implies $n=1$ and $p=q$. 
\end{definition}
By definition, loops are prime cycles, and prime cycles are primitive cycles. In addition, every cycle admits a unique decomposition into a product of prime cycles, which is obtained by dividing the given cycle into cycles at the place where the cycle hits its source. Note that loops, cycles, and primitivity of cycles are standard terminologies for graphs, while prime cycles are usually called first return loops in the standard terminology of symbolic dynamics. 

Let $E_\cycle^*\subseteq E^*$ denote the set of cycles in $E$, and let $E_\prim^*\subseteq E^*$ denote the set of primitive cycles in $E$. We let $[E_\cycle^*]$ and $[E_\prim^*]$ denote the sets of the equivalence classes of cycles and primitive cycles under cyclic permutations. For a cycle $w \in E^*_\cycle$, let $[w]$ denote its equivalence class. Note that if $p \in E^*_{\prim}$, then every $q \in [p]$ is also a primitive cycle. 

Later, we will show a technical variant of the next proposition (see Lemma~\ref{lem:general-cyclic}), so we omit its proof:  
\begin{proposition} \label{prop:general-cyclic}
    For a primitive cycle $p \in E^*_\prim$, we have $|[p]|=|p|$. 
\end{proposition}

\begin{definition}
    Let $n \in \Zz_{\geq 1}$. 
    \begin{enumerate}
        \item The \emph{subdivided circle}, denoted by $S^1_n$, is the graph defined as follows: 
              The vertex set is $(S^1_n)^0=\{v_i \mid i \in \Zz/n\Zz\}$, the edge set is $(S^1_n)^1=\{e_i \mid i \in \Zz/n\Zz\}$, and we have $s(e_i)=v_i, r(e_i)=v_{i+1}$. 
        \item The \emph{$n$-rose}, denoted by $R_n$, is the graph consisting of a unique vertex and $n$ loops at the vertex. 
    \end{enumerate}
\end{definition}

\begin{definition}
    Let $E$ be a finite graph. For $e \in E^1$, let $E \setminus e$ denote the graph whose vertex set is $E^0$ and whose edge set is $E^1 \setminus \{e\}$. 
    A finite graph $E$ is said to be \emph{2-edge-connected} if $E$ is strongly connected and $E \setminus e$ is strongly connected for all $e \in E^1$. 
\end{definition}

\begin{definition}
    Let $N \in \Zz_{>0}$. 
    The \emph{$N$-th higher edge graph} of $E$, denoted by $E^{[N]}$, is the graph whose vertex set is $E^{N-1}$ and whose edge set is $E^N$, with the range and source maps given by $s(w)=w_{[0,N-2]}$ and $r(w)=w_{[1,N-1]}$ for $w \in E^N= (E^{[N]})^1$. 
\end{definition}

\begin{proposition}\label{prop:2edge}
    If $E$ is 2-edge-connected and $E \neq R_2$, then for every $(e,f)\in E^2$, the graph $E^{[2]}\setminus ef$ is strongly connected.
\end{proposition}
\begin{proof}
    Let $(e,f) \in E^2$. It suffices to show that there exists a path $w \in E^*$ which begins with $e$, ends with $f$, and does not go through $ef$. If $E=R_n$ with $n \geq 3$, then this is obvious. We may assume that $|E^0| \geq 2$. Let $v=r(e)=s(f)$, and choose $v' \in E^0 \setminus \{v\}$. Choose a path $w_1 \in (E\setminus f)^*$ from $v$ to $v'$. Similarly, choose a path $w_2 \in (E\setminus e)^*$ from $v'$ to $v$. Then, the path $ew_1w_2f$ does not go through $ef$. 
\end{proof}

We shall use the \emph{periodic decomposition} of strongly connected graphs. 
Recall that a graph $E$ is said to be \emph{primitive} if there exists $N >0$ such that for all $v_1,v_2 \in E^1$, there is a path of length $N$ from $v_1$ to $v_2$. Primitive graphs are strongly connected. Conversely, for a strongly connected graph $E$, there exists an integer $d \geq 1$, which is called a \emph{period} of $E$, and a decomposition 
\[ E^0 = E^0_0 \sqcup \dots \sqcup E^0_{d-1} \]
such that if $e \in E^1$ and $s(e) \in E^0_{i}$, then $r(e) \in E^0_{i+1}$, where indices are considered modulo $d$. Moreover, the graph $E_0$ is primive, where $E_0$ is the graph whose vertex set is $E^0_0$ and whose edge set is
\[ E^1_0 := \{ w \in E^d \mid s(w) \in E^0_0\}. \]
The period $d$ is equal to the GCD of the lengths of all cycles in $E$. 
See \cite[Section~4.5]{LM21} for the details.

\subsection{Shifts of finite type} \label{ssec:SFT}
Let $E$ be a strongly connected graph which is not a subdivided circle. Let 
\[ X_E := \left\{ x =(e_i)_{i=0}^\infty \in \prod_{i=0}^\infty E^1  \mid (e_i, e_{i+1}) \in E^2 \text{ for all } i \in \Nz \right\}. \]
The $i$-th coordinate of $x$ is denoted by $x_{[i]}$.  
Define $\sigma_E \colon X_E \to X_E$ by $\sigma_E(x)_{[i]} = x_{[i+1]}$. Similarly, let 
\[ \overline{X}_E := \left\{ \bar{x} =(e_i)_{i=-\infty}^\infty \in \prod_{i=-\infty}^\infty E^1  \mid (e_i, e_{i+1}) \in E^2 \text{ for all } i \in \Zz \right\}. \]
We use the same notation for the $i$-th coordinate. 
Define $\bar{\sigma}_E \colon \overline{X}_E \to \overline{X}_E$ by $\bar{\sigma}_E(\bar{x})_{[i]} = \bar{x}_{[i+1]}$.  
\begin{definition}
    The semigroup dynamical system $(X_E,\sigma_E)$ is called the \emph{one-sided shift of finite type (SFT)} attached to the graph $E$. The dynamical system $(\overline{X}_E,\bar{\sigma}_E)$ is called the \emph{two-sided shift of finite type (SFT)} attached to the graph $E$. 
\end{definition}
The canonical projection $\ol{X}_E\to X_E$ is denoted by $\pi_E$. For $x \in X_E$ and $i<j$, let $x_{[i,j]}$ denote the finite path $x_{[i]}\dots x_{[j]} \in E^*$, and let $x_{[i,\infty)}$ denote the infinite path $x_{[i]}x_{[i+1]}\dots$. We use the same notation for $\bar{x} \in \ol{X}_E$.

For $w \in E^*$, its cylinder set is denoted by
\[ \cZ(w) = \{ x \in X_E \mid x_{[0,|w|-1]}=w\}. \]

Note that every edge shift can be viewed as a vertex shift (see \cite[Section~2.2.2]{MatBook}).  
Explicitly, define an $E^1\times E^1$ matrix $A$ by $A_{e,f}=1$ if and only if $r(e)=s(f)$. Then, $(X_E,\sigma_E)\cong (X_A,\sigma_A)$.
We will freely use results for vertex shifts from 
\cite{MatBook} 
without reference from here on.

Using the periodic decomposition of $E$, there is a decomposition 
\[ X_E = X_0 \sqcup\dots\sqcup X_{d-1}\]
such that $\sigma_E(X_i)=X_{i+1}$ (the indices are modulo $d$) and the subshift  $(X_0,\sigma_E^d)$ is isomorphic to $(X_{E_0}, \sigma_{E_0})$. Concretely, we have 
\[ X_i = \{ x \in X_E \mid s(x_{[0]}) \in E^0_i\}.\]
We have a similar decomposition for two-sided shifts. 
In this article, the periodic decomposition is used when we apply classical results for $\Aut(\overline{X}_E, \bar{\sigma}_E)$. Namely, we shall use Sigmund's specification \cite{Sig74} and a theorem by Boyle--Krieger \cite{BK87}. 

\subsection{Ample groupoids, groupoid homology, and cohomology}
The reason why we need groupoids for the study of the outer groups of Higman--Thompson groups is concentrated in Proposition~\ref{prop:aut}, which is essentially due to Matui. We shall identify the outer groups of Thompson-like groups with a quotient of the groups of COEs on SFTs, and the main tool for the study will be COEs. In this subsection, we collect the minimum necessary definitions and facts for groupoids. 

A \emph{groupoid} $\cG$ is a small category such that all morphisms are invertible. The object set is denote by $\cG^{(0)}$, and the morphism set is denoted by $\cG^{(1)}$. We regard $\cG^{(0)} \subseteq \cG^{(1)}$ and $\cG^{(1)}=\cG$ by identifying objects and their identity morphisms. For $g \in \cG$, its domain and codomain is denoted by $\s(g)$ and $\r(g)$, respectively. The maps $\s, \r \colon \cG \to \cG^{(0)}$ are called the source map and the range map, respectively. 
A \emph{topological groupoid} is a groupoid such that $\s, \r$, multiplication, and taking inverses are all continuous. In this article, groupoids are always assumed to be second countable, locally compact, and Hausdorff. 

A groupoid is said to be \emph{\'etale} if $\s$ and $\r$ are local homeomorphisms. An \'etale groupoid $\cG$ is \emph{ample} if $\cG^{(0)}$ is totally disconnected. An \'etale groupoid $\cG$ is \emph{minimal} if $\cG x$ is dense for all $x$, where $\cG x = \{ \r(g) \mid g \in \cG, \s(g)=x\}$. An \'etale groupoid $\cG$ is \emph{essentially principal} or \emph{effective} if the interior of ${\rm Iso}(\cG)$ is $\cG^{(0)}$, where ${\rm Iso}(\cG)=\{g \in \cG \mid \s(g)=\r(g)\}$. 
In this article, groupoids are always assumed to be ample, minimal, effective, and $\cG^{(0)}$ is compact. 

For the details of the basic theory of \'etale groupoids, see \cite[Chapter~5]{MatBook}. 

A clopen subset $U \subseteq \cG$ is called a \emph{bisection} if $\s|_U$ and $\r|_U$ are homeomorphisms onto their ranges. A bisection $U$ is called a \emph{global bisection} if $\s(U)=\r(U)=\cG^{(0)}$. 

\begin{definition} \cite[Definition~2.3]{Mat12}
    For a groupoid $\cG$, the \emph{topological full group} $\F(\cG)$ of $\cG$ is the group consisting of all global bisections. 
\end{definition}
The group operation is defined as follows: For $U,V \in \F(\cG)$, the product $UV$ is defined by
\[ UV := \{ gh \mid g \in U, h \in V, \s(g)=\r(h)\} \]
and the inverse $U^{-1}$ is defined by 
\[ U^{-1} := \{ g^{-1} \mid g \in U\}. \]
Since $\cG$ is assumed to be second countable, $\F(\cG)$ is a countable group. 

We shall use the 1st groupoid cohomology (\cite[Definition~9.1.5]{MatBook}).  
Let $\Hom(\cG,\Zz)$ be the abelian group of continuous groupoid homomorphisms from $\cG$ to $\Zz$, and put $\partial\colon C(\cG^{(0)},\Zz)\to \Hom(\cG,\Zz)$ by $\partial(\xi)(g)\coloneqq \xi(\r(g))-\xi(\s(g))$. The group $\H^1(\cG)$ is then 
\[
\H^1(\cG)\coloneqq \Hom(\cG,\Zz)/\{\partial(\xi): \xi\in C(\cG^{(0)},\Zz)\}.
\]
A groupoid isomorphism $\alpha \colon \cG \to \cH$ induces the isomorphism $\H^1(\alpha) \colon \H^1(\cH) \to \H^1(\cG)$ defined by $[\rho]_1 \mapsto [\rho \circ \alpha]_1$ for $\rho \in \Hom(\cH,\Zz)$. 

The groupoid we shall use is the \emph{Cuntz--Krieger groupoid} $\cG_E$ (\cite[Chapter~5.3]{MatBook}). 
For a strongly connected graph $E$ with associated SFT $(X_E,\sigma_E)$, we let 
\begin{equation}
\cG_E\coloneqq\{(x,k-l,y)\in X_E\times\Zz\times X_E : x,y\in E_E, k,l\in\Nz,\text{ and }\sigma_E^k(x)=\sigma_E^l(y)\}.
\end{equation}
The groupoid operation is $(x,n,y)(y,m,z)=(x,n+m,z)$. The range and source maps of $\cG_E$ are given by $\mfr((x,k-l,y))=x$ and $\mfs((x,k-l,y))=y$. The sets 
\begin{equation}
\cZ(U,k,l,V)\coloneqq \{(x,k-l,y)\in\cG_E : x\in U,y\in V, \sigma_E^k\vert_U\text{ is injective, and }\sigma_E^l\vert_V\text{ is injective}\}
\end{equation}
form a basis for an \'etale topology on $\cG_E$, which makes $\cG_E$ an ample groupoid. 
Note that this topology is different from the product topology. 

The groupoid $\cG_E$ is effective and minimal if $E$ is strongly connected and $E \not\cong S^1_n$ (see, e.g.,  \cite[Proposition 5.3.2]{MatBook}). 

\subsection{Higman--Thompson groups} \label{ssec:HT}
For $r,n\in\Zz_{>0}$ with $n\geq 2$, consider the $r\times r$ matrix
\begin{equation}
M=
\begin{bmatrix}
0 & 0 & \cdots & 0 & n \\
1 & 0 & \cdots & 0 & 0 \\
0 & 1 & \cdots & 0 & 0 \\
\vdots & \vdots & \ddots & \vdots & \vdots \\
0 & 0 & \cdots & 1 & 0
\end{bmatrix},
\end{equation}
and let $E$ be the directed graph with $r$ vertices whose adjacency matrix is $M$. Matui showed in \cite[Section~6.7.1]{Mat15} that $\F(\cG_E)$ is isomorphic to the \emph{Higman--Thompson group} $V_{n,r}$, as defined in \cite{Higman74}. 
Consequently, if $E=R_n$ for $n \geq 2$, then $\F(\cG_E)$ is isomorphic to $V_n:=V_{n,1}$.

\subsection{Continuous orbit equivalences of SFTs}
Let $E,F$ be strongly connected graphs which are not subdivided circles. 
The following notion is due to Matsumoto.
\begin{definition}[{\cite[Chapter~4]{MatBook}}]
A homeomorphism $\varphi\colon X_E\to X_F$ is a \emph{continuous orbit equivalence (COE)} if there exists continuous maps $k,l\colon X_E\to\Nz$ and $k',l'\colon X_F\to\Nz$ such that
\begin{equation}
\label{eqn:COE}
    \sigma_F^{k(x)}(\varphi(\sigma_E(x)))=\sigma_F^{l(x)}(\varphi(x))\quad \text{ for all }x\in X_E,
\end{equation}
and 
\begin{equation}
\label{eqn:COE'}
    \sigma_E^{k'(x)}(\varphi^{-1}(\sigma_F(x)))=\sigma_E^{l'(x)}(\varphi^{-1}(x))\quad \text{ for all }x\in X_F.
\end{equation}
\end{definition}

The pair $(k,l)$ is called an \emph{$\varphi$-cocycle pair}. 
\begin{definition}
We let $\COE(X_E,\sigma_E)$ denote the group of all COEs from $X_E$ to $X_E$.    
\end{definition}
The composition of two COEs is again a COE. In particular, 
$\COE(X_E,\sigma_E)$ is indeed a group under composition (\cite[Lemma~4.1.7]{MatBook}). 
Given $x\in X_E$, we let $[x]_\tail$ denote the tail equivalence class of $x$, i.e.,
\begin{equation}
    [x]_\tail\coloneqq \{y\in X_E : \text{ there exist }n,m\geq 0\text{ with }\sigma_E^n(x)=\sigma_E^m(y)\}.
\end{equation}

A point $x\in X_E$ is \emph{eventually periodic} if there exists natural numbers $n\neq m$ such that $\sigma_E^n(x)=\sigma_E^m(x)$. 
Let $X_E^{ep}$ denote the subset of $X_E$ consisting of all eventually periodic points. Any $x\in X_E^{ep}$ can be written as $x=wp^\infty$ for some $p\in E_\prim^*$ and $w \in E^*$.
For $p,q\in E_\prim^*$ and $v,w \in E^*$, we have $[vp^\infty]_\tail=[wq^\infty]_\tail$ if and only if $[p]=[q]$, so there is a bijection 
\begin{equation}
[E_\prim^*]\to \{[x]_\tail : x\in X_E^{ep} \},\quad [p]\mapsto [p^\infty]_\tail.
\end{equation} 
For every COE $\varphi \colon X_E \to X_F$, we have $\varphi(X_E^{ep})=X_F^{ep}$ (\cite[Proposition~4.2.4]{MatBook}). Hence, a COE $\varphi \colon X_E \to X_F$ induces a bijection $[E^*_\prim] \to [F^*_\prim]$, which we shall denote by $[p]\mapsto \varphi[p]$. We have  $\varphi[p]=[q]$ if and only if $\varphi([p^\infty]_\tail)=[q^\infty]_\tail$.
In particular, we have a canonical action $\COE(X_E,\sigma_E)\acts [E_\prim^*]$.  

\subsection{Ordered cohomology for SFTs}
\label{ss:cohomology}
Let $E,F$ be strongly connected graphs which are not subdivided circles. 
The ordered cohomology group of the one-sided SFT $(X_E,\sigma_E)$ is defined in \cite{MM14} as the analogue of the ordered cohomology for two-sided SFTs from \cite{BH96,Poon89}. It is the ordered group $(\H^E,\H^E_+)$, where
\begin{equation}
\H^E\coloneqq C(X_E,\Zz)/\{f-f\circ\sigma_E :f\in C(X_E,\Zz)\},
\end{equation}
and the positive cone is defined as 
$\H^E_+\coloneqq \{[f] \in \H^E : f\geq0\}$. 
There is a homomorphism $\Phi_E\colon \Hom(\cG_E,\Zz)\to C(X_E,\Zz)$ defined by 
\[
\Phi_E(\rho)(x)\coloneqq\rho(x,1,\sigma_E(x))
\]
for all $\rho \in \Hom(\cG_E,\Zz)$. The map $\Phi_E$ descends to an isomorphism $\Phi_E\colon\H^1(\cG_E)\overset{\cong}{\to} \H^E$ (\cite[Proposition~9.1.8]{MatBook}). 
The isomorphism $\Phi_E\colon\H^1(\cG_E)\overset{\cong}{\to} \H^E$ induces an action $\COE(X_E,\sigma_E)\acts \H^E$ by order isomorphisms. Explicitly, 
by \cite[Proposition~9.3.8]{MatBook}, 
each COE $\varphi \colon X_E \to X_F$ gives rise to an invertible continuous linear map of Banach spaces $\Psi_\varphi\colon C(X_F,\Cz)\to C(X_E,\Cz)$ by
\begin{equation}
\label{eqn:Psi}
    \Psi_\varphi(f)(x)\coloneqq\sum_{i=0}^{l(x)-1}f(\sigma_F^i(\varphi(x)))-\sum_{j=0}^{k(x)-1}f(\sigma_F^j(\varphi(\sigma_E(x))))
\end{equation}
for all $x\in X_E$, where $(k,l)$ is any $\varphi$-cocycle pair. Further, $\Psi_\varphi(C(X_F,\Zz))=C(X_E,\Zz)$, and $\Psi_\varphi$ descends to an isomorphism of ordered groups
\begin{equation}
    \ol{\Psi}_\varphi\colon (\H^F,\H^F_+)\overset{\cong}{\to} (\H^E,\H^E_+),\quad [f]\mapsto [\Psi_\varphi(f)],
\end{equation}
see \cite[Theorem~9.3.11]{MatBook}. 
Then, we have $\ol{\Psi}_\varphi = \Phi_E \circ \H^1(\alpha) \circ \Phi_F^{-1}$, where $\alpha \in \Out(\cG_E)$ is such that $\alpha^0=\varphi$. 
Now the action $\COE(X_E,\sigma_E)\acts \H^E$ is given by $\varphi\mapsto\overline{\Psi}_\varphi^{-1}$.

A measure $\mu$ on $X_E$ is said to be \emph{shift-invariant} if $(\sigma_E)_*\mu=\mu$, where $(\sigma_E)_*$ denotes the push-forward. Following the notation from \cite{MM16},
we let $M(X_E,\sigma_E)_+$ be the cone of regular Borel shift-invariant measures on $X_E$, and let $P(X_E,\sigma_E)$ be the simplex of regular Borel shift-invariant probability measures on $X_E$. 
 
Recall that a homomorphism $\phi\colon \H^E\to\Rz$ is said to be \emph{positive} if $\phi(\H^E_+)\subseteq \Rz_{\geq 0}$ (cf. \cite{Good86}). Each $\mu\in M(X_E,\sigma_E)_+$ gives rise to a positive homomorphism $\H^E\to\Rz$ by $\mu(f)=\int_{X_E}f\,d\mu$, and Riesz representation theorem implies that this assignment gives an isomorphism of $M(X_E,\sigma_E)_+$ with the cone of positive homomorphisms of $\H^E$ (cf. \cite[Proposition~2.1]{Poon89} where this is observed for two-sided SFTs). Combining this with \cite[Lemma 9.1.1]{MatBook}, we see that the canonical projection map $\pi_E\colon\ol{X}_E\to X_E$ induces an isomorphism $P(\ol{X}_E,\ol{\sigma}_E)\cong P(X_E,\sigma_E)$, where $P(\ol{X}_E,\ol{\sigma}_E)$ is the simplex of shift-invariant regular Borel probability measures on $\ol{X}_E$.

\begin{definition}
    Given a class of a primitive cycle $[p]\in [E_\prim^*]$, we let 
\begin{equation}
   \eta_{[p]}\coloneqq  \sum_{i=0}^{|p|-1}\delta_{\sigma_E^i(p^\infty)}
\end{equation}
be the associated (unnormalised) shift-invariant measure in $M(X_E,\sigma_E)_+$ supported on the tail equivalence class of $p^\infty$. Such measures are called \emph{periodic measures}. 
\end{definition}

\section{Outer automorphism groups of ample groupoids} \label{sec:outer}

\subsection{The definition of $\Out(\cG)$} 
For a groupoid $\cG$, we let $\Aut(\cG)$ denote the group of homeomorphic automorphisms of $\cG$. For a bisection $U\subseteq\cG$, put $\s_U\coloneqq\s\vert_U$ and $\r_U\coloneqq \r\vert_U$.

There is an injective homomorphsim $\Ad\colon\F(\cG)\to \Aut(\cG)$ given by 
\begin{equation*}    
\Ad(U)(g)=Ug U^{-1}\coloneqq (\s_U^{-1}(\r(g)))g (\r_U^{-1}(\s(g)))
\end{equation*}
for all  $U\in\F(\cG)$ and $g\in\cG$. A short calculation shows that $\alpha\Ad(U)=\Ad(\alpha(U))\alpha$ for all $\alpha\in\Aut(\cG)$ and $U\in\F(\cG)$. In particular, 
the image of $\F(\cG_E)$ is a normal subgroup of $\Aut(\cG)$.

\begin{definition}
We let $\Inn(\cG)$ denote the image of $\F(\cG)$ in $\Aut(\cG)$ under $\Ad$. Elements of $\Inn(\cG)$ are called \emph{inner automorphisms} of $\cG$. 
The \emph{outer automorphism group} of $\cG$ is $\Out(\cG)\coloneqq \Aut(\cG)/\Inn(\cG)$.
\end{definition}
For a bisction $U \subseteq \cG$, let $\pi_U \colon \s(U) \to \r(U)$ denote the map 
$\pi_U := (\r_U)\circ (\s_U)^{-1}$. For $U \in \F(\cG)$, we have $\Ad(U)(x)=\pi_U(x)$ for $x \in \cG^{(0)}$ by definition.

For $\alpha \in \Aut(\cG)$, let $\alpha^0$ denote the restriction of $\alpha$ to $\cG^{(0)}$. 
The next proposition implies that $\alpha$ is inner if and only if $\alpha$ is homologically similar to the identity in the sense of \cite[Definition~3.4]{Mat15}. 

\begin{proposition}
\label{prop:inner2}
    Let $X=\cG^{(0)}$, and let $\alpha\in\Aut(\cG)$. Then, $\alpha$ is inner if and only if there exists a 
    continuous open map $b \colon X \to \cG$ such that $\s(b(x))=x$ and $\r(b(x))=\alpha^0(x)$ for all $x \in X$. 
\end{proposition}
\begin{proof}
    Suppose $\alpha$ is inner, and let $\alpha = \Ad (U)$ for a full bisection $U \in \F(\cG)$. By definition, we have $\alpha^0 = \pi_U = (\r_U) \circ (\s_U)^{-1}$, so that $b := (\s_U)^{-1}$ satisfies the desired property. 
    Conversely, let $b \colon X \to \cG$ be a continuous map as in the statement. Then, $U=\{b(x) \mid x \in X\}$ is a global bisection such that $\alpha = \Ad(U)$. 
\end{proof}

For Cuntz--Krieger groupoids, we have the following well-known result, which is contained in, e.g., the proof of \cite[Theorem~9.2.2]{MatBook}.
\begin{proposition}
\label{prop:Aut=COE}
For all $\alpha\in\Aut(\cG_E)$, we have $\alpha^0\in\COE(X_E,\sigma_E)$, and the map $\Aut(\cG_E)\to \COE(X_E,\sigma_E)$ given by $\alpha\mapsto\alpha^0$ is an isomorphism.
\end{proposition}

We say that a COE $\varphi$ is \emph{inner} if it is equal to $\alpha^0$ for some inner automorphism $\alpha\in\Inn(\cG_E)$. 

\begin{proposition}
    Inner COEs act trivially on the set of all tail equivalence classes of $X^{ep}_E$. In particular, we have a canonical action $\Out(\cG_E) \acts [E^*_\prim]$.  
\end{proposition}
\begin{proof}
    For every $g \in \cG_E$ with $\s(g) \in X^{ep}_E$, $\s(g) \in X^{ep}_E$ and $\r(g) \in X^{ep}_E$ are tail equivalent by definition. For a global bisection $U \subseteq \cG_E$, we have $\pi_U(x) = \r(g)$, where $g \in U$ is the unique element with $\s(g)=x$. Hence, inner automorphisms of $\cG_E$ acts trivially on the classes of $X^{ep}_E$.  
\end{proof}

The following lemma is essentially well-known: 
\begin{lemma} \label{lem:existinner}
    If $x_1,x_2 \in X_E$ are tail equivalent, then there exists an inner COE $\varphi$ such that $\varphi(x_1)=x_2$. 
\end{lemma}
\begin{proof}
    Let $x_1,x_2 \in X_E$ be tail equivalent infinite paths. Then, we can choose $w_1,w_2 \in E^*$ and $y \in X_E$ such that $x_1=w_1y$ and $x_2=w_2y$. By choosing $w_1$ and $w_2$ larger, if necessary, we may assume that $\cZ(w_1) \cap \cZ(w_2) = \emptyset$. Define $\varphi \colon X_E \to X_E$ by 
    $\varphi(w_1z)=w_2z$ on $\cZ(w_1)$, $\varphi(w_2z)=w_1z$ on $\cZ(w_2)$, and $\varphi(z)=z$ on $X_E \setminus (\cZ(w_1)\cup \cZ(w_2))$. Then, we can see that $\varphi$ is an inner COE with $\varphi(x)=y$. Indeed, we can see that $\varphi=\Ad(U)^0$, where the bisection $U$ is given by
    \[ U = \{ (w_2z,|w_2|-|w_1|,w_1z) \mid s(z)=s(y)\} \cup 
    \{ (w_1z,|w_1|-|w_2|,w_2z) \mid s(z)=s(y)\} \cup 
    (X_E \setminus (\cZ(w_1)\cup \cZ(w_2)).
    \]
\end{proof}

Later, we shall prove that a COE is inner if and only if it preserves all tail equivalence classes of eventually periodic points (see Proposition~\ref{prop:preservesorbits->inner}).

\subsection{The action of $\Out(\cG)$ on groupoid cohomology}
Let us record the following basic observation for easy reference.
\begin{proposition}
\label{prop:actionsoncohom}
There is a group homomorphism $\Out(\cG)\to\Aut(\H^1(\cG))$ given $[\alpha]\mapsto \H^1(\alpha^{-1})$.
\end{proposition}
\begin{proof}
We need to show that if $\alpha\in\Inn(\cG)$, then $\H^1(\alpha)=\id$. Let $\rho\in\Hom(\cG,\Zz)$, and write $\alpha=\Ad(U)$ for some $U\in\F(\cG)$. For $g\in\cG$, we have $\alpha(g)=Ug U^{-1}=(\s\vert_U)^{-1}(\r(g))g ((\s\vert_U)^{-1}(\s(g)))^{-1}$. Hence, 
\begin{equation*}
\rho\circ\alpha(g)=\rho((\s\vert_U)^{-1}(\r(g)))+\rho(g)-\rho((\s\vert_U)^{-1}(\s(g))),    
\end{equation*}
so that
\begin{align*}
(\rho\circ\alpha-\rho)(g)
=\rho\circ\alpha(g)-\rho(g)
=\rho((\s\vert_U)^{-1}(\r(g)))-\rho((\s\vert_U)^{-1}(\s(g)))
=\partial(\rho\circ(\s\vert_U)^{-1})(g)
\end{align*}
which implies that $[\rho\circ\alpha]=[\rho]$ in $\H^1(\cG)$.
\end{proof}

\subsection{Outer automorphism groups of topological full groups}

Each $\alpha\in\Aut(\cG)$ gives rise to a group automorphism $\alpha_{\F}\in\Aut(\F(\cG))$ satisfying $\alpha_{\F}(U)\coloneqq\alpha(U)$ for all $U\in\F(\cG)$, and the map $\Aut(\cG)\to\Aut(\F(\cG))$ given by $\alpha\mapsto\alpha_{\F}$ is clearly a group homomorphism. The next proposition is essentially due to Matui \cite[Section~3]{Mat15}. 
\begin{proposition} \label{prop:aut}
 Assume $\cG^{(0)}$ is a Cantor set. 
 Then, 
 \begin{enumerate}
     \item the homomorphism $\Aut(\cG)\to\Aut(\F(\cG))$, $\alpha\mapsto\alpha_{\F}$, is an isomorphism;
     \item the map from (i) descends to a group isomorphism $\Out(\cG)\cong\Out(\F(\cG))$.
 \end{enumerate}
\end{proposition}
\begin{proof}
(ii) easily follows from (i), so it suffices to prove (i). 
For every $g \in \cG$, the set 
\[ T(g) = \{ U \in \F(\cG) \mid U^2 = \cG^{(0)}, U \setminus \cG^{(0)} \ni g\}
\]
forms a basis of open neighbourhoods of $g$. Since we have 
\[ \alpha(g) = \alpha\left( \bigcap_{U \in T(g)}U\right) = \bigcap_{U \in T(g)} \alpha_{\F}(U) \]
for every $\alpha \in \Aut(\cG)$, the map $\alpha \mapsto \alpha_{\F}$ is injective. Let $\Phi \in \Aut(\F(\cG))$. By \cite[Theorem~3.5]{Mat15}, there exists $\varphi \in \Homeo(\cG^{(0)})$ such that $\pi_V = \varphi\circ\pi_U\circ\varphi^{-1}$, where $V=\Phi(U)$, for all $U \in \F(\cG)$. 
By the proof of \cite[Proposition~3.8]{Mat15}, there exists a unique $\alpha \in \Aut(\cG)$ such that $\pi_{\alpha(U)} = \varphi\circ\pi_U\circ\varphi^{-1}$ for all $U \in \F(\cG)$. Then, we have $\alpha_{\F} = \Phi$ by definition, which completes the proof of surjectivity. 
\end{proof}

\section{Projectivised measures and marker COEs}
In this section, we define a space analogous to the projectivised current space of a free group \cite{Kap06, Uyanik14}, and we introduce the concept of a marker COE, inspired by marker automorphisms of shifts of finite type.

\subsection{The space of projectivised measures} 
\label{ssec:PM+}

Let us now define what will be our boundary action.
\begin{definition}
\label{def:PM+}
    We let \begin{equation*}
        \Pz M(X_E,\sigma_E)_+\coloneqq (M(X_E,\sigma_E)_+\setminus\{0\})/\sim,
    \end{equation*}
where $\mu_1\sim\mu_2$ if there exists $r\in\Rz_{>0}$ such that $\mu_1=r\mu_2$. 
\end{definition}
For $\mu\in M(X_E,\sigma_E)_+$, we let $[\mu]$ denote the class of $\mu$ in $\Pz M(X_E,\sigma_E)_+$.
Note that $\Pz M(X_E,\sigma_E)_+$ is homeomorphic to $P(X_E,\sigma_E)$. 
In particular, $\Pz M(X_E,\sigma_E)_+$ is a compact Hausdorff space by Alaoglu's theorem. 

\begin{definition}
    The actions $\Out(\cG_E)\acts M(X_E,\sigma_E)_+$ and $\Out(\cG_E)\acts \Pz M(X_E,\sigma_E)_+$ are defined as the ones induced from the action $\Out(\cG_E)\acts (\H^E,\H^E_+)$ by order isomorphisms arising from \cite[Theorem~9.3.11]{MatBook}.
\end{definition}
For $\alpha \in \Aut(\cG_E)$ and $\mu \in M(X_E,\sigma_E)_+$, the actions are denoted by $\alpha(\mu) \in M(X_E,\sigma_E)_+$ and $\alpha[\mu] \in \Pz M(X_E,\sigma_E)_+$. We use the same notations for $\varphi \in \COE(X_E,\sigma_E)$. The actions are  concretely given by 
\[ \varphi(\mu) := \mu \circ \Psi_\varphi \text{ and } \varphi[\mu] := [\mu \circ \Psi_\varphi]\]
for $\varphi \in \COE(X_E,\sigma_E)$ and $\mu \in M(X_E,\sigma_E)_+$,. 

\begin{remark}
    The action $\Out(\cG_E) \acts M(X_E,\sigma_E)_+$ is completely different from the push-forward of measures. Indeed, for $\varphi \in \COE(X_E,\sigma_E)$, $\Psi_\varphi\colon M(X_E,\sigma_E)_+\to M(X_E,\sigma_E)_+$ does not preserve $P(X_E,\sigma_E)$ unless $\varphi$ is a strong COE (see \cite[Theorem~7.2]{MM16}). 
\end{remark}

\begin{remark} \label{rem:EF}  
    Let $E,F$ be strongly connected graphs which are not subdivided circles. 
    A COE $X_E \to X_F$ also induces the homeomorphisms $M(X_E,\sigma_E)_+ \to M(X_F,\sigma_F)_+$ and $\Pz M(X_E,\sigma_E)_+ \to \Pz M(X_F,\sigma_F)_+$ in the same way. 
\end{remark}

Our aim now is to prove that $\Out(\cG_E)\acts \Pz M(X_E,\sigma_E)_+$ is a topologically free boundary action. Our first observation is that the projective measure space is a compactification of $[E^*_\prim]$. 
This is a consequence of \cite[Theorem~1]{Sig74}. 

\begin{proposition}   
\label{prop:periodicmeasures}
    If $E$ is strongly connected, then $\{[\eta_{[p]}] : [p]\in [E_\prim^*]\}$ is dense in $\Pz M(X_E,\sigma_E)_+$.
\end{proposition}
\begin{proof}
We first claim that we may assume $E$ is primitive. 
Let $E^0=E^0_0\sqcup\cdots \sqcup E^0_{d-1}$ be the periodic decomposition of $E$ and let $X_E = X_0 \sqcup\cdots \sqcup X_{d-1}$ be the associated decomposition.  
Then, there is an isomorphism of simplices of probability measures $P(X_E,\sigma_E)\cong P(X_0,\sigma_E^d) = P(X_{E_0}, \sigma_{E_0})$. Indeed, for $\mu\in P(X_E,\sigma_E)$, the measure $d\mu\vert_{X_0}$ lies in $P(X_{E_0},\sigma_{E_0})$, and the map $\mu\mapsto d\mu\vert_{X_0}$ is invertible with inverse $\nu\mapsto \frac{1}{d}\sum_{i=0}^{d-1}(\sigma_E^i)_*\nu$. This bijection is an isomorphism of simplices that identifies the sets of periodic measures. Thus, we may assume that $E$ is primitive. 

As observed in Section~\ref{ss:cohomology}, the canonical projection map $\pi_E\colon \ol{X}_E\to X_E$ induces an isomorphism of simplices $P(\ol{X}_E,\ol{\sigma}_E)\cong P(X_E,\sigma_E)$. Clearly this map identifies the set of probability measures on $\ol{X}_E$ coming from finite $\ol{\sigma}_E$-orbits with the set $\{\frac{1}{|p|}\eta_{[p]} : [p]\in [E_\prim^*]\}$. Thus, it suffices to show that the probability measures on $\ol{X}_E$ coming from finite $\ol{\sigma}_E$-orbits are dense in the set of all shift-invariant probability measures on $\ol{X}_E$.
It is known that two-sided SFTs for primitive graphs have specification, so the result follows from \cite{Sig74}.
\end{proof}

\begin{proposition}
    \label{prop:varphieta[p]}
Let $E,F$ be strongly connected graphs which are not subdivided circles. 
Let $\alpha \colon \cG_E \to \cG_F$ be an isomorphism of groupoids. For every primitive cycle $p\in E_\prim^*$, we have 
\begin{equation}
\alpha(\eta_{[p]})=\eta_{\alpha[p]}.   
\end{equation}
\end{proposition}

\begin{proof}
The map $\alpha \colon M(X_E,\sigma_E)_+ \to M(X_F,\sigma_F)_+$ is characterised as follows: for $\mu\in M(X_E,\sigma_E)_+$, we have 
\begin{equation}
\alpha(\mu)
=\mu\circ (\Phi_E\circ \H^1(\alpha)\circ \Phi^{-1}_F),
\end{equation}
where $\Phi_E\colon \H^1(\cG_E)\to \H^E$ is the isomorphism from Section~\ref{ss:cohomology}.
Hence,
\begin{equation}
\alpha(\mu)(\Phi_F(\rho))=\int_{X_E}\Phi_E(\rho\circ\alpha)\,d\mu
=\int_{X_E}(\rho\circ\alpha)(x,1,\sigma_E(x))\,d\mu
\end{equation}
for all $\rho\in\Hom(\cG_F,\Zz)$. Let $p\in E_\prim^*$, and take $q\in F_\prim^*$ with $\alpha[p]=[q]$. Then, we can write $\alpha^{0}(p^\infty)=wq^\infty$ for some $w\in E^*$. By \cite[Remark~4.6 and Proposition~4.8]{CEOR19}, we have $\alpha(p^\infty,|p|,p^\infty)=(wq^\infty,|q|,wq^\infty)$. Hence,
\begin{align}
  \alpha(\eta_{[p]})(\Phi_F(\rho))&=\sum_{i=0}^{|p|-1}(\rho\circ\alpha)(\sigma_E^i(p^\infty),1,\sigma_E^{i+1}(p^\infty))\\
  &=\rho\left(\alpha(p^\infty,|p|,p^\infty)\right)=\rho(wq^\infty,|q|,wq^\infty)=\rho(q^\infty,|q|,q^\infty)=\eta_{[q]}(\Phi_F(\rho))
\end{align}
for all $\rho\in\Hom(\cG_F,\Zz)$. Here, we used that $(q^\infty,|q|,q^\infty)$ and $(wq^\infty,|q|,wq^\infty)$ are conjugate in $\cG_E$, so that $\rho(wq^\infty,|q|,wq^\infty)=\rho(q^\infty,|q|,q^\infty)$.
We have shown that $\alpha(\eta_{[p]})=\eta_{[q]}$.
\end{proof}

We establish basic properties of the topology of the projective measure space. They will be frequently used in Section~\ref{ssec:proximal}. 
\begin{definition}
For $w\in E^*$ and $\mu\in M(X_E,\sigma_E)_+$, we let 
\[\langle w,\mu\rangle\coloneqq \mu(\cZ(w)). \] 
For $[\mu]\in \Pz M(X_E,\sigma_E)_+$, $\varepsilon > 0$, and $R \in \Zz_{>0}$, let \begin{equation*}
U([\mu],\varepsilon, R)\coloneqq\left\{[\nu]\in \Pz M(X_E,\sigma_E)_+ : \left| \frac{\langle w,\nu\rangle}{\|\nu\|} - \frac{\langle w,\mu\rangle}{\|\mu\|}
    \right| < \varepsilon \text{ for all }w \in E^* \text{ with }|w| \leq R\right\}.
\end{equation*} 
\end{definition}

\begin{lemma} 
\label{lem:current-topology}
For every $[\mu]\in \Pz M(X_E,\sigma_E)_+$, the family $\{U([\mu],\varepsilon, R)\}_{\varepsilon,R}$ is a neighbourhood basis of $[\mu]$. Moreover, there is a continuous embedding $\Pz M(X_E,\sigma_E)_+\hookrightarrow M(X_E,\sigma_E)_+$ given by $[\nu]\mapsto \nu/||\nu||$.
\end{lemma}
\begin{proof}
    The continuous map $[\nu]\mapsto \nu/||\nu||$ is a homeomorphism onto its range $P(X_E,\sigma_E)$. Since characteristic functions on cylinder sets span a dense subspace of $C(X_E, \Cz)$, the open sets 
    \begin{equation*}
        V(\mu,\varepsilon, R)\coloneqq\left\{\nu\in M(X_E,\sigma_E)_+ : \left| \langle w,\nu\rangle - \langle w,\mu \rangle \right| < \varepsilon \text{ for all }w \in E^* \text{ with }|w| \leq R\right\}
    \end{equation*} 
    for $\varepsilon>0$ and $R>0$ 
    form a neighbourhood basis of $\mu \in P(X_E,\sigma_E)$ by definition of weak*-topology. 
    Hence, the claim follows.
\end{proof}

\begin{lemma} \label{lem:subword}
    Let $\mu \in M(X_E,\sigma_E)_+$ and let $v,w \in E^*$. If $w$ contains a $v$-segment, then 
    \[ \langle w ,\mu \rangle \leq \langle v,\mu \rangle\]
\end{lemma}
\begin{proof}
    Write $w=w_1vw_2$ for $w_1,w_2 \in E^*$, and let $n=|w_1|$. By shift-invariance of $\mu$, we have 
    \[ \langle v,\mu \rangle = \sum_{w' \in E^n, r(w')=s(v)} \langle w'v, \mu\rangle \geq \langle w_1v,\mu \rangle
    \geq \langle w_1vw_2,\mu \rangle. \]
\end{proof}

\begin{definition} \label{def:occurrence}
    Let $v\in E^*$, $w \in E^*_\prim$,
    and write 
    $[w] = \{w(1), \dots,w(n)\}$.   
    Let $\langle v,[w]\rangle$ denote the number of $k \in \{1,\dots,n\}$ such that $w(k)$ begins with $v$. 
\end{definition}
Note that $\langle v,[w]\rangle$ is nearly equal to the number of $v$-segments in $w$, but another $v$-segment can be found by taking cyclic permutations. 
The next lemma is immediate from the definition:

\begin{lemma} \label{lem:rational-current}
    Let $v\in E^*$, $w \in E^*_\prim$. Suppose $|v| \leq |w|$. Then, 
    we have 
    \[ \frac{\langle v,\eta_{[w]} \rangle}{\| \eta_{[w]}\|} = \frac{\langle v,[w]\rangle}{|w|}.  \]
\end{lemma}

\begin{lemma} \label{lem:seminorms}
    For every $v \in E^*$, the function $\Pz M(X_E,\sigma_E)_+\to\Rz_{\geq 0}$ given by 
    \[ [\nu] \mapsto \frac{\langle v, \nu \rangle}{\| \nu\|}\]
    is continuous. 
\end{lemma}
\begin{proof}
    Clearly, it is continuous on $M(X_E,\sigma_E)_+$, and its value does not depend on the choice of representative of $[\nu]$. So it induces a continuous map on $\Pz M(X_E,\sigma_E)_+$. 
\end{proof}

\begin{lemma} \label{lem:embedding}
    Let $N \in \Zz_{>0}$, $w \in E^N$. Then, the following hold: 
    \begin{enumerate}
        \item There is a canonical embedding $\Pz M(X_{E^{[N]}\setminus w},\sigma_{E^{[N]}\setminus w})_+ \to \Pz M(X_E,\sigma_E)_+$. 
        \item In (1), the image is equal to the classes $[\mu]$ such that $\langle w,\mu \rangle =0$. 
    \end{enumerate}
\end{lemma}
\begin{proof}
    (1): 
    First, observe that the SFT $(X_{E^{[N]}}, \sigma_{E^{[N]}})$ is canonically identified with $(X_E, \sigma_E)$ via the map 
    \[ x=x_{[0]}x_{[1]}x_{[2]} \dots \mapsto (x_{[0]})_{[0]}(x_{[1]})_{[0]}(x_{[2]})_{[0]}\dots. \]
    Note that $(x_{[i]})_{[0]}$ denotes the first edge of the path $x_{[i]} \in E^N$ by definition. 
    Under this identification, the subshift $(X_{E^{[N]}\setminus w},\sigma_{E^{[N]}\setminus w})$ is identified with 
    \[ Y := \{ x \in X_E \mid x_{[k,k+N-1]} \neq w \text{ for all } k \geq 0\}. \]
    That is, $(X_{E^{[N]}\setminus w},\sigma_{E^{[N]}\setminus w}) \cong (Y,\sigma_E)$. Hence, the embedding is induced from push-forward from the inclusion map $Y\to X_E$, which is shift-equivariant. 

    (2):  
    For $\mu \in M(X_E,\sigma_E)_+$ with $\langle w, \mu\rangle = 0$, we have $\langle w',\mu\rangle = 0$ for all $w' \in E^*$ which contain $w$ by Lemma~\ref{lem:subword}. It follows that $\mu$ is concentrated on $X_E\setminus\bigcup_{w'}Z(w')$, where the union runs over all words $w'$ containing $w$, which is precisely $Y$. 
\end{proof}

Lemma~\ref{lem:embedding} shall be used in the following way, except for the case $E=R_2$: 
\begin{corollary}\label{cor:rational-approx}
Let $(e,f)\in E^2$. Suppose that $E^{[2]} \setminus ef$ is strongly connected. If $[\mu] \in \Pz M(X_E,\sigma_E)_+$ satisfies $\langle ef,\mu \rangle =0$, then $[\mu]$ can be approximated by projective classes $[\eta_{[w]}]$ such that $w$ does not contain $ef$. 
\end{corollary}
\begin{proof}
    This follows from Lemma~\ref{lem:embedding} and Proposition~\ref{prop:periodicmeasures}. 
\end{proof}

\subsection{Marker COEs} \label{ssec:marker}
The notion of a marker automorphism for an SFT goes back to \cite{BLR88}, see also \cite{Ash90} and \cite{BFK90}. 
We modify the usual definition by allowing the data set to consist of two words of (possibly) different lengths. 
Let $D=\{d,d'\}\subseteq E^* \cup \{o\}$, where $d\neq d'$ (the \emph{data}) such that $s(d)=s(d')=v$ and $r(d)=r(d')=u$. Let $\m, \m'\in E^*$ (the \emph{markers}) be such that $s(\m')=u$ and $r(\m)=v$. We assume that either $\m=\m'$, or $\m \neq \m'$ and $|\m|=|\m'|=1$. 

We want to define a map $\varphi\colon X_E\to X_E$ by replacing each instance of $\m d\m'$ in $x$ with $\m d'\m'$ and each instance of $\m d'\m'$ in $x$ with $\m d\m'$. We now introduce overlap conditions which ensure such a map is bijective.

\begin{definition}
    For $w,w' \in E^*$, \emph{the set of overlaps} of $w$ with $w'$ is defined by
    \[ S(w,w') = \{ k \in \Zz_{>0} \colon w_{[|w|-k,|w|-1]} = w'_{[0,k-1]} \}.
    \]
\end{definition}

In the rest of this subsection, we assume the following, which we call the \emph{overlap conditions} (cf. \cite[Section~13.2]{LM21}): 
\begin{itemize}
    \item $\m d \m'$ is not a subword of $\m d' \m'$, and $\m d' \m'$ is not a subword of $\m d \m'$. 
    \item If $\m=\m'$, then 
    \begin{align*}
        &S(\m d \m,\m d' \m),\  S(\m d' \m,\m d \m) \subseteq \{1,\dots,|\m|\}, \\
        &S(\m d \m,\m d \m)\subseteq\{1,\dots,|\m|\}\cup\{|\m d \m|\},\\
        &S(\m d' \m,\m d' \m)\subseteq\{1,\dots,|\m|\} \cup\{|\m d' \m|\}.
    \end{align*} 
    \item If $\m \neq \m'$ and $\m,\m' \in E^1$, then
    \begin{align*}
        &S(\m d \m',\m d' \m') = S(\m d' \m',\m d \m') = \emptyset,\\
        &S(\m d \m',\m d \m')=\{|\m d \m'|\},\\
        &S(\m d' \m',\m d' \m')=\{|\m d' \m'|\}.
    \end{align*}
\end{itemize}

\begin{lemma} \label{lem:disjoint}
    We have $\cZ(\m d\m')\cap \cZ(\m d'\m')=\emptyset$. 
\end{lemma}
\begin{proof}
    Since $S(\m d \m',\m d' \m'),\  S(\m d' \m',\m d \m') \subseteq \{1,\dots,|\m|\}$, there does not exist $w \in E^*\cup\{o\}$ such that $\m d\m'=\m d'\m' w$ or $\m d \m' w = \m d' \m'$. Hence, $\cZ(\m d\m')\cap \cZ(\m d'\m')=\emptyset$.
\end{proof}

We define $\varphi\colon X_E\to X_E$ as follows: given $x\in X_E$, we inductively define sequences $w_n\in E^*$ and $y_n\in X_E$ by 
\begin{enumerate}
    \item $w_0\coloneqq o$, and $y_0\coloneqq x$;
    \item \begin{itemize}
        \item if $y_n\in\cZ(\m d\m')$, then $w_{n+1}\coloneqq \m d'$ and $y_{n+1}$ is defined to be the unique infinite path such that $y_n=\m d y_{n+1}$;
        \item if $y_n\in\cZ(\m d'\m')$, then $w_{n+1}\coloneqq \m d$ and $y_{n+1}$ is defined to be the unique infinite path such that $y_n=\m d' y_{n+1}$;
        \item if $y_n\in X_E\setminus \cZ(\m d\m')\cup \cZ(\m d'\m')$, then $w_{n+1}\coloneqq (y_n)_{[0]}$ and $y_{n+1}$  defined as the unique infinite path such that $y_n=w_{n+1}y_{n+1}$.
    \end{itemize}
\end{enumerate}
This process is well-defined by Lemma~\ref{lem:disjoint}. 
Now we put $\varphi(x)\coloneqq w_1w_2\cdots$. Note that $\varphi$ is characterised by 
\begin{equation*}
    \varphi(x)=\begin{cases}
        \m d'\varphi(\m' y) & \text{ if } x=\m d \m' y \text{ for some } y \in X_E,\\
        \m d\varphi(\m' y) & \text{ if } x=\m d' \m' y \text{ for some } y \in X_E,\\
        x_{[0]}\varphi(x_{[1,\infty)}) & \text{ if }x\in X_E\setminus \cZ(\m d\m')\cup \cZ(\m d'\m').
    \end{cases}
\end{equation*}

\begin{definition}
    We call the COE $\varphi$ defined above a \emph{marker COE} for the markers $\m, \m'$ and the data set $D=\{d,d'\}$. We call it \emph{type I} if $\m=\m'$, and \emph{type II} if $\m \neq \m'$ and $|\m|=|\m'|=1$. 
\end{definition}

\begin{proposition} \label{prop:bijective}
    The marker COE $\varphi$ is an involutive homeomorphism. 
\end{proposition}
\begin{proof}
    We first show continuity of $\varphi$. Let 
    \[ c = \frac{\min\{|\m d \m'|, |\m d'\m'|\}}{\max\{|\m d \m|, |\m d' \m'|\}}.\]
    For $x,y \in X_E$, if we have $x_{[0,n-1]} = y_{[0,n-1]}$ for some $n \in \Zz_{>0}$, then $\varphi(x)_{[0,m-1]}=\varphi(x)_{[0,m-1]}$, where $m$ is the integer part of $cn$, by definition of $\varphi$. Therefore, if $x_n \in X_E$ is a sequence converging to $x \in X_E$, then $\varphi(x_n) \to \varphi(x)$ as $n \to \infty$. This shows continuity. 
    
    Hence, it suffices to show $\varphi^2=\id$. 
    Let $x = \m d \m' y$ for some $y \in X_E$. Then, we have  
    $\varphi(x) = \m d' \varphi(\m' y)$. 
    We claim that $\varphi(\m'y) \in \cZ(\m')$. If it is of type II,  
    then we have $\m'y \not\in \cZ(\m d \m') \cup \cZ(\m d' \m')$, so that $\varphi(\m'y) = \m' \varphi(y)$. 
    Suppose it is of type I. Decompose $\m=\m_1\m_2$ with $\m_1, \m_2 \in E^* \cup \{o\}$ such that 
    $\m_{[k,|\m|-1]} y \not\in \cZ(\m d\m) \cup \cZ(\m d' \m)$ for all $0 \leq k \leq |\m_1|-1$ and 
    $\m_2 y \in \cZ(\m d\m) \cup \cZ(\m d' \m)$. 
    If $\m_2=o$, then $\varphi(\m y)=\m \varphi(y) \in \cZ(\m)$ by definition of $\varphi$. Otherwise, 
    $\m_2$ is a prefix of $\m$. Hence, we have $\varphi(\m y) = \m_1 \varphi(\m_2y)$ and $\varphi(\m_2 y) \in \cZ(\m) \subseteq \cZ(\m_2)$, so that $\varphi(\m y) \in \cZ(\m_1\m_2)=\cZ(\m)$ by definition of $\varphi$. 
    Therefore, we have $\varphi(x) \in \cZ(\m d' \m')$, so that $\varphi^2(x) = \m d \varphi^2(\m'y)$. Similarly, if $x = \m d' \m' y$ for some $y \in X_E$, then $\varphi^2(x) = \m d' \varphi^2(\m'y)$. 

    Let $x\in X_E\setminus \cZ(\m d\m')\cup \cZ(\m d'\m')$. We show that $\varphi(x) \in  X_E\setminus \cZ(\m d\m')\cup \cZ(\m d'\m')$. Suppose $\varphi(x) = \m d \m' y$ for some $y \in X_E$. 
    If $x_{[k,\infty)} \not\in \cZ(\m d\m')\cup \cZ(\m d'\m')$ for all $0 \leq k \leq |\m d \m'|-1$, then 
    \[ \varphi(x) = x_{[0,|\m d \m'|-1]} \varphi(x_{[|\m d \m'|-1, \infty)}) = \m d \m'y,\] so that 
    $x_{[0,|\m d \m'|-1]} = \m d \m'$, which is a contradiction. 
    Hence, we can decompose $x = w z$ such that $0 < |w| < |\m d \m'|$,  
    $x_{[k,\infty)} \not\in \cZ(\m d\m')\cup \cZ(\m d'\m')$ for all $0 \leq k \leq |w|-1$, and $z \in \cZ(\m d\m')\cup \cZ(\m d'\m')$. Then, we have $\varphi(x)=w\varphi(z)$ and $\varphi(z) \in \cZ(\m d\m')\cup \cZ(\m d'\m')$. If $\varphi$ is of type II, then this contradicts the overlap conditions. Suppose $\varphi$ is of type I. Then, we see that $|w| < |\m d|$ because if $|w| \geq |\m d|$, then $\varphi$ does not change the first $w$-segment of $x$ and the subsequent $\m$-segment of $x$, so that $x$ begins with $\m d \m$ since $|w|+|\m| \geq |\m d \m|$, which is a contradiction. Hence, the overlap is bigger than $|\m|$, which is a contradiction. 
    Thus, $\varphi(x) \not\in \cZ(\m d\m')$. We see $\varphi(x) \not\in \cZ(\m d'\m')$ in the same way, so that $\varphi(x) \in  X_E\setminus \cZ(\m d\m')\cup \cZ(\m d'\m')$. Therefore, $\varphi^2(x) = x_{[0]}\varphi^2(x_{[1,\infty]})$. 

    Overall, we obtain that 
    \begin{equation*}
         \varphi^2(x)=\begin{cases}
            \m d\varphi^2(\m' y) & \text{ if } x=\m d \m' y \text{ for some } y \in X_E,\\
            \m d'\varphi^2(\m' y) & \text{ if } x=\m d' \m' y \text{ for some } y \in X_E,\\
            x_{[0]}\varphi^2(x_{[1,\infty]}) & \text{ if }x\in X_E\setminus \cZ(\m d\m')\cup \cZ(\m d'\m').
        \end{cases}
    \end{equation*}
    Therefore, we conclude that $\varphi^2=\id$ by taking the limit. 
\end{proof}

\begin{proposition}
    The marker COE $\varphi$ is indeed a COE.
\end{proposition}
\begin{proof}
Let $m=|\m|=|\m'|$. 
By Lemma~\ref{lem:disjoint}, we can define maps $k,l\colon X_E\to \Zz_{\geq 0}$ by 
\[
k(x)\coloneqq \begin{cases}
 m-1+|d|   & \text{ if } x\in \cZ(\m d\m'),\\
 m-1+|d'|   & \text{ if } x\in \cZ(\m d'\m'),\\
  0  & \text{ if } x\in X_E\setminus \cZ(\m d\m')\cup \cZ(\m d'\m'),
\end{cases}
\]
and 
\[
l(x)\coloneqq \begin{cases}
m+|d'|   & \text{ if } x\in \cZ(\m d\m'),\\
m+|d|   & \text{ if } x\in \cZ(\m d'\m'),\\
 1   & \text{ if } x\in X_E\setminus \cZ(\m d\m')\cup \cZ(\m d'\m').
\end{cases}
\]
Then, $k,l$ are locally constant, hence continuous. We will show that $\sigma_E^{k(x)}(\varphi(\sigma_E(x))=\sigma_E^{l(x)}(\varphi(x))$ for all $x\in X_E$. Since $\varphi$ is an involution, this suffices to show that $\varphi$ is a COE. 

If $x\in \cZ(\m d\m')$, then we can write $x=\m d\m' y$, and we have $\varphi(x)=\m d'\varphi(\m' y)$. 
On the other hand, we have $\sigma_E(x)=w\m'y$, where $w=\m_{[1,m-1]}d$. 
The overlap conditions imply that for every $0 \leq k \leq |w|-1$, $w_{[k,|w|-1]}\m'$ is not a prefix of $\m d \m'$ or $\m d' \m'$, and $\m d' \m'$ is not a prefix of $w_{[k,|w|-1]}\m'$. 
Hence, $\sigma_E(x)_{[k,\infty)} \not\in  \cZ(\m d\m')\cup \cZ(\m d'\m')$ for every $0 \leq k \leq |w|-1$, so that $\varphi(\sigma_E(x))=\m_{[1,m-1]}d\varphi(\m' y)$. 
Now we see that $\sigma_E^{m-1+|d|}(\varphi(\sigma_E(x)))=\sigma_E^{m+|d'|}(\varphi(x))$. 
The same argument works for $x\in \cZ(\m d'\m')$. 
If $x\in X_E\setminus \cZ(\m d\m')\cup \cZ(\m d'\m')$, then we can write $x=ey$ for some $e\in E^1$ such that $\varphi(x)=e\varphi(y)$. Now we have $\varphi(\sigma_E(x))=\varphi(y)=\sigma_E(e\varphi(y))=\sigma_E(\varphi(x))$.
\end{proof}

\subsection{Good representatives for marker COEs} \label{ssec:good}

In this subsection, we consider a type I or type II marker COE $\varphi$ for the markers $\m,\m' \in E^1$ and $D=\{o,v\}$ with $v \in E^*$. 

\begin{definition}
    Let $[w] \in [E_\cycle^*]$. 
    \begin{enumerate}
        \item Suppose $\varphi$ is of type I.  
        \begin{itemize}
            \item When a representative of $[w]$ contains a $\m \m$-segment or a $\m v \m$-segment, a representative $w \in [w]$ is called a \emph{good representative} if $w$ starts with $\m$ and ends with $\m$ or $\m v$. 
            \item When $[w]=[\m]$, the unique representative $\m \in [\m]$ is called a \emph{good representative}. When $[w]=[\m v]$, the representative $\m v \in [\m v]$ is called a \emph{good representative}. 
            \item When $[w] \neq [\m], [\m v]$, and every representative $[w]$ contains no $\m \m$-segment and no $\m v \m$-segment, every representative $w \in [w]$ is called a \emph{good representative}. 
        \end{itemize}
        \item Suppose $\varphi$ is of type II.  
        \begin{itemize}
            \item When a representative of $[w]$ contains a $\m \m'$-segment or a $\m v \m'$-segment, a representative $w \in [w]$ is called a \emph{good representative} if $w$ ends with $\m\m'$ or $\m v \m'$.
            \item When every representative of $[w]$ contains no $\m \m'$-segment and no $\m v \m'$-segment, every representative $w \in [w]$ is called a \emph{good representative}.
        \end{itemize}
    \end{enumerate}
\end{definition}

For a marker COE $\varphi$ as above, we associate a map $F_\varphi \colon [E_\cycle^*] \to [E_\cycle^*]$ in the following way: Let $[w] \in [E^*_\cycle]$ and let $w$ be a good representative of $[w]$. 

If $\varphi$ is of type II, then
\begin{enumerate}
    \item Let $w_0\coloneqq o$ and let $y_0\coloneqq w$;
    \item \begin{itemize}
        \item if $y_n$ starts with $\m v \m'$, then $w_{n+1}\coloneqq \m$ and $y_{n+1}$ is defined to be the unique finite path such that $y_n=\m v y_{n+1}$;
        \item if $y_n$ starts with $\m \m'$, then $w_{n+1}\coloneqq \m v$ and $y_{n+1}$ is defined to be the unique finite path such that $y_n=\m y_{n+1}$;
        \item if $y_n$ is non-empty and does not start with $\m\m'$ nor $\m v \m'$, then $w_{n+1}$ is defined to be the first edge of $y_n$, and $y_{n+1}$ is defined to be the unique finite (or empty) path such that $y_n=w_{n+1}y_{n+1}$.
        \item if $y_n=o$, then put $N=n$ and we finish this procedure. 
    \end{itemize}
\end{enumerate}
Then, we put $\bar{F}_\varphi(w)=w_1\dots w_N$. 

If $\varphi$ is of type I, then 
\begin{enumerate}
    \item If $w \neq \m,\m v$ and $w$ does not contain any $\m\m$-segment or any $\m v \m$-segment, we let $\bar{F}_{\varphi}(w)=w$ and finish the procedure. 
    \item Otherwise, let $w_0\coloneqq o$ and let $y_0\coloneqq w \m$;
    \item \begin{itemize}
        \item if $y_n$ starts with $\m v \m$, then $w_{n+1}\coloneqq \m$ and $y_{n+1}$ is defined to be the unique finite path such that $y_n=\m v y_{n+1}$;
        \item if $y_n$ starts with $\m \m$, then $w_{n+1}\coloneqq \m v$ and $y_{n+1}$ is defined to be the unique finite path such that $y_n=\m y_{n+1}$;
        \item if $y_n \neq \m$ and does not start with $\m\m$ nor $\m v \m$, then $w_{n+1}$ is defined to be the first edge of $y_n$, and $y_{n+1}$ is defined to be the unique finite path such that $y_n=w_{n+1}y_{n+1}$.
        \item if $y_n=\m$, then put $N=n$ and we finish this procedure. 
    \end{itemize}
\end{enumerate}
Then, we put $\bar{F}_\varphi(w)=w_1\dots w_N$. Note that $\bar{F}_\varphi(\m)=\m v$ and $\bar{F}_{\varphi}(\m v) = \m$. 

We define $F_\varphi \colon [E_\cycle^*] \to [E_\cycle^*]$ by $F_\varphi([w]) = [\bar{F}_\varphi(w)]$ for $[w] \in [E_\cycle^*]$ with a good representative $w \in [w]$. Note that if $w_1,w_2 \in [w]$ are both good representatives, then it is not difficult to see that $\bar{F}(w_1)$ and $\bar{F}(w_2)$ are equivalent under cyclic permutations. 

\begin{remark}
    If $w$ is a good representative of $[w]$, then $\varphi(w^\infty)=\bar{F}_{\varphi}(w)^\infty$ by construction. 
    This is because by the overlap conditions, $\varphi$ does not cause any replacement at the intermediate of $ww$. 
\end{remark}

\begin{lemma} \label{lem:goodrep}
    If $[w] \in [E^*_{\prim}]$, then $F_\varphi([w]) \in [E^*_{\prim}]$. 
    In particular, 
    $\varphi[w]=F_\varphi([w])$ for every $[w] \in [E^*_\prim]$. 
\end{lemma}
\begin{proof}
    Let $w \in [w]$ be a good representative. Then, $\bar{F}_{\varphi}(w)$ is again a good representative of $F_{\varphi}([w])$. Hence, $F^2_{\varphi}([w]) = [\bar{F}^2_{\varphi}(w)]$. Now we can see that $\bar{F}^2_{\varphi}(w)=w$ in the same way as Proposition~\ref{prop:bijective}. By the overlap conditions, we see that $\bar{F}_{\varphi}(w^n)=\bar{F}_{\varphi}(w)^n$. Hence, we have $F_{\varphi}([E^*_\cycle]\setminus[E^*_\prim]) \subseteq [E^*_\cycle]\setminus[E^*_\prim]$. Since $F^2_{\varphi}=\id$, we conclude that $F_{\varphi}([E^*_\cycle]\setminus[E^*_\prim]) = [E^*_\cycle]\setminus[E^*_\prim]$ and $F_{\varphi}([E^*_\prim]) = [E^*_\prim]$. 
\end{proof}

\section{Faithfulness and Transitivity} \label{sec:faithful-transitive}
We are now going to prove that the action $\Out(\cG_E) \acts \Pz M(X_E,\sigma_E)_+$ is a topologically free boundary action. The proof consists of 3 steps: faithfulness of $\Out(\cG_E)\acts [E^*_\prim]$, transitivity of $\Out(\cG_E) \acts [E^*_\prim]$, and strong proximality of $\Out(\cG_E) \acts \Pz M(X_E,\sigma_E)_+$. The first two claims shall be proved in this section.

\subsection{Faithfulness}
In this subsection, we prove that the action $\Out(\cG_E) \acts [E^*_\prim]$ is faithful. 
A homeomorphism $\bar{\varphi} \colon \overline{X}_E \to \overline{X}_E$ is called a \emph{conjugacy} if $\bar{\varphi}\circ\bar{\sigma}_E=\bar{\sigma}_E \circ \bar{\varphi}$. 
The next lemma is a variant of \cite[Theorem~2.5]{BK87}:
\begin{lemma} \label{lem:BoyleKrieger}
    Let $E$ be a strongly connected graph with $E \not\cong S^1_n$ for any $n \geq 1$. 
    Let $E^0 = E^0_0  \sqcup \dots \sqcup E^0_{d-1}$ be the periodic decomposition, and let $\overline{X}_E = \overline{X}_0 \sqcup\dots\sqcup \overline{X}_{d-1}$ be the associated decomposition. 
    Let $\bar{\varphi} \colon \overline{X}_E \to \overline{X}_E$ be a conjugacy. Suppose $\bar{\varphi}(\bar{x}) \in \bar{\sigma}_{E}^{\Zz}(\bar{x})$ for all periodic points $\bar{x} \in \overline{X}_{E}$. Then, there exists a locally constant function $L\colon X_E \to \Zz$ such that $\bar{\varphi}(\bar{x})=\bar{\sigma}^{L(x)}(\bar{x})$ for all $\bar{x} \in \overline{X}_E$, where $x = \pi_E(\bar{x})$. 
\end{lemma}
\begin{proof}
    Since $\varphi$ commutes with the shift, it suffices to show that there is a locally constant function $L \colon X_0 \to \Zz$ such that 
    $\bar{\varphi}(\bar{x})=\bar{\sigma}^{L(x)}(\bar{x})$ for all $\bar{x} \in \overline{X}_0$, where $X_0 = \pi_E(\overline{X}_0)$ and $x = \pi_E(\bar{x})$. 

    Let $\overline{Y}_i := \overline{X}_0 \cap \bar{\varphi}^{-1}(\overline{X}_i)$. Then, we have $\overline{X}_0 = \overline{Y}_0 \sqcup \dots \sqcup \overline{Y}_{d-1}$ and $\bar{\varphi}(\overline{Y}_i) \subseteq \ol{X}_i$. In addition, $\overline{Y}_i$ is $\bar{\sigma}_E^{d}$-invariant. Define $n(\bar{x}) \equiv i$ on $\overline{Y}_i$. 
    We claim that $\bar{\sigma}^{-i}(\bar{\varphi}(\overline{Y}_i)) = \overline{Y}_i$. Let $\bar{x} \in \overline{Y}_i$ be a periodic point. Then, since $\bar{x}$ and $\bar{\varphi}(\bar{x})$ belongs to the same orbit of $\sigma$, we have 
    $\bar{\varphi}(\bar{x}) = \bar{\sigma}_E^{i+dk}(\bar{x})$ for some $k$. 
    Hence, we have $\bar{\sigma}^{-i}(\bar{\varphi}(\bar{x})) \in \overline{Y}_i$ by the $\bar{\sigma}^d_E$-invariance of $\ol{Y}_i$. Hence, $\bar{\sigma}^{-i}(\bar{\varphi}(\overline{Y}_i)) \subseteq \overline{Y}_i$ by denseness of periodic points. 
    Conversely, let $\bar{y} \in \overline{Y}_i$ be a periodic point, and let $\bar{x} = \bar{\varphi}^{-1}(\sigma^{i}(\bar{y}))$. 
    Then, since $\bar{\varphi}(\bar{y}) = \bar{\sigma}_E^{i+dk}(\bar{y})$ for some $k$, 
    we have $\bar{\varphi}^{-1}(\bar{y}) = \bar{\sigma}_E^{-i-dk}(\bar{y})$, so that 
    $\bar{x} = \bar{\sigma}_E^{-dk}(\bar{y}) \in \overline{Y}_i$ by the $\bar{\sigma}^d_E$-invariance of $\ol{Y}_i$. 
    Since, $\bar{\sigma}^{-i}\bar{\varphi}(\bar{x})=\bar{y}$, we have $\bar{\sigma}^{-i}(\bar{\varphi}(\overline{Y}_i)) \supseteq \overline{Y}_i$. 
    Now we have proved the claim. Let $\bar{\psi}(\bar{x}) = \bar{\sigma}^{-n(\bar{x})}(\bar{\varphi}(\bar{x}))$ for $\bar{x} \in \overline{X}_0$. Then, $\bar{\psi}$ is a homeomorphism from $\overline{X}_0$ to $X_0$ which commutes with $\bar{\sigma}^d_E$. 
    Then, there exists $r \in \Zz$ such that $\bar{\psi}=\bar{\sigma}_E^{dr}$ on $\ol{X}_0$ by \cite[Theorem~2.5]{BK87}. Therefore, $\bar{\varphi}(\bar{x}) = \bar{\sigma}^{L(\bar{x})}(\bar{x})$ for $\bar{x} \in \overline{X}_0$, where $L(\bar{x}) = n(\bar{x}) + dr$. 
    
    It remains to show that $L$ factors through $\pi_E$. Let $\bar{x} \in \overline{Y}_i,\bar{y} \in \overline{Y}_j$ such that $\bar{x}_{[0,\infty)}=\bar{y}_{[0,\infty)}$. 
    Since all $\overline{Y}_i$ are open, $\bar{\sigma}^{Nd}_E(\bar{x})$ and $\bar{\sigma}^{Nd}_E(\bar{y})$ belong to the same $\overline{Y}_k$ for large $N$. By the $\bar{\sigma}^d_E$-invariance of $\overline{Y}_k$, we have $k=i=j$. Hence, $L(\bar{x})=L(\bar{y})$ by definition of the map $n$. 
\end{proof}

Following Matsumoto \cite[Section~10.2]{MatBook}  
we say that a COE is a \emph{strong COE} if the induced map on the cohomology group $\H^E$ preserves the canonical order unit, i.e., if $\overline{\Psi}_\varphi([1_{X_E}])=[1_{X_E}]$. Note that $\Psi_\varphi(1_{X_E})=l-k$, where $(k,l)$ is any $\varphi$-cocycle pair.

We shall need Matsumoto's construction of a two-sided conjugacy from a strong COE, which is given in the proof of \cite[Proposition~10.2.10]{MatBook}. 

Suppose we have a strongly continuous orbit equivalence $\varphi\colon X_E\to X_E$ with $\varphi$-cocycle pair $(k,l)$. This means that we can write $c\coloneqq l-k=1+b-b\circ\sigma_E$ for some continuous function $b\colon X_E\to\Nz$, 
see \cite[Lemma~10.2.4]{MatBook} and \cite[Definition~10.2.1]{MatBook}. 
From the proof of 
\cite[Proposition~10.2.10]{MatBook}, 
we obtain a conjugacy $\bar{\varphi}\colon\ol{X}_E\to\ol{X}_E$ such that 
 \[
 \bar{\varphi}(\bar{x})_{[j,\infty)}=\sigma_E^{b(\bar{x}_{[j,\infty)})}(\varphi(\bar{x}_{[j,\infty)}))
 \]
for all $\bar{x}\in \ol{X}_E$ and $j\in\Zz$, where 
we regard $\bar{x}_{[j,\infty)} \in X_E$ as a one-sided infinite sequence starting from $j$-th entry of $\bar{x}$.

\begin{proposition}
\label{prop:preservesorbits->inner}
Let $E$ be a strongly connected graph with $E \not\cong S^1_n$ for any $n \geq 1$. 
If $\varphi \in \COE(X_E,\sigma_E)$ satisfies  $\varphi([x]_\tail)=[x]_\tail$ for all $x \in X^{ep}$, then $\varphi$ is inner. 
\end{proposition}
\begin{proof}
We first show that $\varphi$ must be a strong COE. For every primitive cycle $p$, we use \cite[Theorem 9.3.11]{MatBook} and Proposition~\ref{prop:varphieta[p]} to obtain
\[
\eta_{[p]}([c])=\eta_{[p]}(\Psi_\varphi[1_{X_E}])=\eta_{\varphi[p]}([1_{X_E}])=\eta_{[p]}([1_{X_E}]),
\]
so that $\eta_{[p]}([c]-[1_{X_E}])=0$. It now follows from \cite[Lemma~9.1.3]{MatBook} that $[c]-[1_{X_E}]=0$. This shows that $\varphi$ is a strong COE.

Let $\bar{\varphi}\colon\ol{X}_E\to\ol{X}_E$ be the conjugacy constructed in the paragraphs preceding this proposition. We will show that $\bar{\varphi}(\bar{x})\in \ol{\sigma}_E^\Zz(\bar{x})$ for every periodic point $\bar{x}\in \ol{X}_E$. 
Such a point can be written as $\bar{x}=p^{\pm\infty}=...ppp...$ with $(p^{\pm\infty})_{[0,|p|-1]}=p$, where $p \in E^*_\prim$ is a primitive cycle. 
Since $\bar{\varphi}$ is a conjugacy, we have $\bar{\varphi}(p^{\pm\infty})=q^{\pm\infty}$, where $q \in E^*_\prim$ with $|p|=|q|$. 
This is because the orbits of $(\ol{X}_E,\bar{\sigma}_E)$ of cardinality $n$ are precisely the orbits of $r^{\pm \infty}$ with $r \in E^*_\prim$ and $|r|=n$. 
By our assumption on $\varphi$, we have $\varphi(p^\infty)=wp^\infty$ for some finite word $w\in E^*\cup\{o\}$. Thus, we have
\[
q^\infty=
(q^{\pm\infty})_{[0,\infty)}=(\bar{\varphi}(p^{\pm\infty}))_{[0,\infty)}=\sigma_E^{b(\bar{x}_{[0,\infty)})}(\varphi(\bar{x}_{[0,\infty)}))=\sigma_E^{b(p^\infty)}(wp^\infty),
\]
which implies that $[q] = [p]$ in $[E^*_\prim]$, so that $\bar{\varphi}(\bar{x})\in \ol{\sigma}_E^\Zz(\bar{x})$. 
By Lemma~\ref{lem:BoyleKrieger}, we choose a locally constant function $L\colon X_E \to \Zz$ such that $\bar{\varphi}(\bar{x})=\bar{\sigma}^{L(x)}(\bar{x})$ for all $\bar{x} \in \overline{X}_E$, where $x = \pi_E(\bar{x})$.

We show that $L$ must be nonnegative. For all $\bar{x}\in\ol{X}_E$ with $x=\pi_E(\bar{x})$, we have 
\begin{equation}
\label{eqn:pospower}
    \bar{x}_{[L(x),\infty)}=\bar{\sigma}_E^{L(x)}(\bar{x})_{[0,\infty)}
    =\bar{\varphi}(\bar{x})_{[0,\infty)}
    =\sigma_E^{b(x)}(\varphi(x)).
\end{equation}
Since the right-hand side of \eqref{eqn:pospower} only depends on $x=\bar{x}_{[0,\infty)}$, this forces $L(x)\geq 0$ for all $x \in X_E$.

Fix $x \in X_E$, and let $\bar{x} \in \ol{X}_E$ be an element such that $\bar{x}_{[0,\infty)} = x$. We have 
\[ \sigma_E^{L(x)}(x) = \bar{\sigma}_E^{L(x)}(\bar{x})_{[0,\infty)} = \bar{\varphi}(\bar{x})_{[0,\infty)} = \sigma_E^{b(x)}(\varphi(x)), \]
which means that $(\varphi(x), b(x)-L(x), x) \in \cG_E$. Then, the map
\[ X_E \to \cG_E,\ x \mapsto (\varphi(x), b(x)-L(x), x)\]
satisfies the condition in Proposition~\ref{prop:inner2}. Hence, $\varphi$ is inner. 
\end{proof}

\begin{corollary}
\label{cor:faithful-on-ep}
The actions $\Out(\cG_E)\acts [E_\prim^*]$ and $\Out(\cG_E)\acts (\H^E,\H^E_+)$ are faithful.
\end{corollary}
\begin{proof}
The first statement is immediate from Proposition~\ref{prop:preservesorbits->inner}.
Suppose $\alpha\in\Aut(\cG_E)$ satisfies $\ol{\Psi}_{\alpha^0}=\id$. 
Then, $\alpha$ fixes all periodic measures, so that $\alpha$ acts trivially on $[E^*_\prim]$ by Proposition~\ref{prop:varphieta[p]}. Hence, $\alpha$ is inner by the first statement. 
\end{proof}

\subsection{Transitivity}

In this section, we prove that the action $\Out(\cG_E) \acts [E^*_\prim]$ is transitive. We fix a strongly connected graph $E$. 

\begin{definition}
    For a cycle $p = p_1 \dots p_n$, where $p_i$ is a prime cycle, let $p[i,j] = p_i \dots p_j$. In addition, let $\ell(p) = n$, which is called the \emph{prime-path-length} of $p$.
\end{definition}

In this section, when we check the primitivity of a given cycle $p$, we can use the prime cycle interval $p[i,j]$ instead of the letter-wise interval $p_{[i,j]}$. This because, if $p = (p'_{[0,k]})^n$ for some $k \geq 0$, $n \geq 2$, then $p'$ is a cycle at $s(p)$, which implies that $p'_{[0,k]}=p[0,l]$ for some $l \geq 0$. Similarly, when we check the overlap conditions for two words $\m d \m$ and $\m d' \m$, where the marker $\m$ is a cycle, we can use the prime cycle intervals. Indeed, suppose we have $(\m d \m)_{[0,k-1]}=(\m d' \m)_{[n-k, n-1]}$, where $n=|\m d' \m|$. Then, we have $s((\m d' \m)_{[n-k]}) = s((\m d \m)_{[0]})=s(\m)$ and $r((\m d \m)_{[k-1]}) = r((\m d' \m)_{[n-1]})=r(\m)$, so that $(\m d \m)_{[0,k-1]} = (\m d \m)[0,l-1]$ and $(\m d' \m)_{[n-k, n-1]} = (\m d' \m)[N-l',N-1]$, where $N=\ell(\m d'\m)$ and $l,l' \geq 0$. In addition, we have $l=l'$ because $(\m d \m)[0,l-1]=(\m d' \m)[N-l',N-1]$ implies that the prime-path-lengths of both sides are the same. 

\begin{lemma} \label{lem:general-cyclic}
    Let $p \in E^*_\cycle$ be a cycle, and let $N = \ell(p)$. Let $1 \leq k \leq N-1$. If we have 
    \[ p[k,N-1]p[0,k-2] = p[0,N-2], \]
    then $p$ is not primitive. Moreover, the least period of $p$ divides ${\rm GCD}(k,N)$. 
\end{lemma}

Note that $p[0,k-2]$ is the empty path if $k=1$. For a cycle $p \in E^*_\cycle$, if we have $p=q^n$ for some $q \in E^*_\prim$, then we say that the least period of $p$ is equal to $\ell(q)$. 

\begin{proof}
    Let $p=p_0\dots p_{N-1}$ be the prime cycle decomposition. 
    By comparing prime factors, the assumption is equivalent to 
    \begin{equation} \label{eqn:cyclicpermut2}
    \begin{cases}
        p_i = p_{i+k} & (0 \leq i \leq N-k-1), \\
        p_i = p_{i+N-k} & (0 \leq i \leq k-2). \\
    \end{cases}        
    \end{equation}
    We consider indices as elements of $\Zz/N\Zz$. 
    Then, Equation~\eqref{eqn:cyclicpermut2} is equivalent to 
    \begin{equation} \label{eqn:cyclicpermut3}
        p_i = p_{i+k} \ (0 \leq i \leq N-2).
    \end{equation} 
    We define an equivalence relation on $\Zz/N\Zz$ by $i \sim j$ if and only if $p_i=p_j$. 
    We first show that the equivalence relation $\sim$ is finer than the orbit equivalence relation of the action $\Zz \acts \Zz/N\Zz$ defined by $x \mapsto x+k$. 
    To see this, it suffices to show that the equation $p_{k-1}=p_{N-1}$ is obtained from Equation~\eqref{eqn:cyclicpermut3}. 
    Take $m \geq 2$ such that $mk = {\rm LCM}(k,N)$, because $k <N$. Then, $k-1, 2k-1, \dots, (m-1)k-1 \in \Zz/N\Zz$ are distinct elements, and none of them is congruent to $N-1$. Hence, for $1 \leq l \leq m-1$, the equation $p_{lk-1}=p_{(l+1)k-1}$ is the case $i\equiv lk-1$ in Equation~\eqref{eqn:cyclicpermut3}.
    Since $p_{mk-1}=p_{N-1}$, the claim follows.     
    
    Let $i={\rm GCD}(k, N)$ and $n=N/i$. Then, we have $p=(p_0\dots p_{i-1})^n$ because the subgroup generated by $k$ in $\Zz/N\Zz$ is the same as the subgroup generated by ${\rm GCD}(N,k)$. Since $n>1$, the cycle $p$ is a non-trivial power, and the least period divides $i$.  
\end{proof}

\begin{lemma} \label{lem:prim-transitivity}
    Let $p,q \in E^*_\prim$ be primitive cycles, and suppose  $\ell(q)=\ell(p)+1$ and $q[0,\ell(p)-1] = p$. Then, there exists $\alpha \in \Aut(\cG_E)$ such that $\alpha(p^\infty)=q^\infty$. 
\end{lemma}
\begin{proof}
    Put $p=p_0\dots p_{N-1}$ and $q=p_0\dots p_{N-1}p_N$, where $N =\ell(p) \geq 1$. Put $\m = p$, and we show that
    the marker $\m$ and the data set $D=\{ o, p_N\}$ satisfy the overlap conditions. First, we see that 
    \[ S(\m\m, \m\m) \subseteq \{1,\dots,|\m|\} \cup \{|\m\m| \} \text{ and } S(\m p_N\m, \m p_N \m) \subseteq \{1,\dots,|\m|\} \cup \{|\m p_N\m|\}\]
    by Proposition~\ref{prop:general-cyclic}. 
    We show that 
    $S(\m p_N\m, \m\m) \subseteq \{1,\dots,|\m|\}$. Suppose there is an overlap bigger than $|\m|$. Then, there exists $1 \leq k \leq N$ such that 
    \[ (\m \m)[0,2N-k] = (\m p_N\m)[k,2N].\]
    Then, by comparing the first $N$ prime cycles, we have 
    \[ p_0 \dots p_{N-1} = p_k \dots p_{N}p_0\dots p_{k-2}. \]
    By Lemma~\ref{lem:general-cyclic}, $q$ is not primitive, which is a contradiction. 
    Hence, $S(\m p_N\m, \m\m) \subseteq \{1,\dots,|\m|\}$. Similarly, suppose there is an overlap of $\m\m$ with $\m p_N\m$ which is bigger than $|\m|$. Then, there exists $1 \leq k \leq N-1$ such that 
    \[ (\m p_N\m)[0,2N-k-1] = (\m\m)[k,2N-1].\]
    By comparing the first $N$ prime cycles, we have 
    \[ p_0 \dots p_{N-1} = p_k \dots p_{N-1}p_0\dots p_{k-1}, \]
    which contradics Proposition~\ref{prop:general-cyclic}. 
    Hence, we have $S(\m \m, \m p_N \m) \subseteq \{1,\dots,|\m|\}$. 
    Finally, $\m\m$ is not a subword of $\m p_N\m$. Indeed, since $\ell(\m\m)=2N$ and $\ell(\m p_N \m)=2N+1$, if $\m\m$ is a subword of $\m p_N \m$, then $\m \m$ is a prefix or a suffix of $\m p_N\m$, which contradics other overlap conditions.    
    Hence, the type I marker COE $\alpha$ corresponding to $\m$ and $D$ satisfies the desired property.
\end{proof}

\begin{definition}
    Let $p,q \in E^*_\prim$. We say that $q$ is a \emph{primitive extension} of $p$ if $p=q[0,\ell(p)-1]$ and $q[0,k-1]$ is primitive for all $\ell(p) \leq k \leq \ell(q)$. 
\end{definition}

\begin{lemma} \label{lem:Fine--Wilf}
    For any primitive cycle $p \in E^*_\prim$, there is at most one prime cycle $q \in E^*_\cycle$ such that $pq^k$ is not primitive for some $k \geq 1$. 
\end{lemma}
\begin{proof}
    Let $A$ be the set of prime cycles which appear in the prime decomposition of $p$. If $q \in \E^*_\cycle$ is a prime cycle at $s(p)$ such that $q \not\in A$, then $pq^k$ is obviously primitive for all $k\geq 1$. 
    Hence, the claim follows from 
    Unary-Extension Lemma by Rytter (\cite[Lemma 2.5]{Rytter16}, see also \cite[p.286]{CLR21}) applied to the alphabet set $A$. 
\end{proof}

\begin{lemma} \label{lem:4cycles}
    The following hold:  
    \begin{enumerate}
        \item For every primitive cycle $p \in E^*_\cycle$, there exists a primitive extension $q \in E^*_\prim$ of $p$ such that $q$ goes through all edges of $E$.
        \item For primitive cycles $p,q \in E^*_\prim$ with $s(p)=s(q)$, there exists $[r] \in [E^*_\prim]$ such that $[r]$ contains a primitive extension of $p$ and a primitive extension of $q$. 
    \end{enumerate}
\end{lemma}
\begin{proof}
    We show (1). 
    If $p$ goes through all edges, then $q=p$ satisfies the desired property. Suppose not, and let $e$ be an edge which does not appear in $p$. Choose $w_1,w_2 \in E^1$ such that $s(w_1)=r(w_2)=s(p)$, $r(w_1)=s(e)$, $s(w_2)=r(e)$, and $w_1,w_2$ do not go through $s(p)$ at any intermediate position. Then, $p_1=w_1ew_2$ is a prime cycle at $s(p)$, and $pp_1$ is a primitive extension of $p$. Indeed, if it is a non-trivial power of a prefix of $p$, then $p_1$ coincides with some prime factor of $p$, which contradicts the fact that $p$ does not contain $e$. 
    Next, we choose $e'$ which does not appear in $pp_1$ and do the same process. We can continue this process until $q=pp_1\dots p_n$ goes through all edges. By construction, $q$ is a primitive extension of $p$.   

    We show (2). We may assume that $\ell(p) \leq \ell(q)$, and $[q]$ does not contain a primitive extension of $p$. 
    By Lemma~\ref{lem:Fine--Wilf}, choose prime cycles $s,t \in E^*_\cycle$ such that $ps^k$ and $qt^l$ are primitive extensions of $p$ and $q$ respectively for all $k,l \geq 1$. Let $p_0,q_0 \in E^*_\cycle$ be prime cycles such that $p_0 \neq s$, $q_0 \neq t$, $p_0$ appears in the prime decomposition of $p$, and $q_0$ appears in the prime decomposition of $q$. Let $n>2\ell(q)$ and $m=n+\ell(q)-\ell(p)$, and put $r=ps^mqt^n$. We have $\ell(r) = 2(n+\ell(q))$. 
    
    We claim that $r$ is a primitive extension of $p$. Suppose $r[0,L-1]$ is not primitive for some $n+\ell(q) < L \leq 2(n+\ell(q))$, and put $r[0,L-1] = r[0,K-1]^T$ for some $K \leq L$ and $T \geq 2$. 
    Then, we have $K \geq \ell(p)$. Indeed, if $K < \ell(p)$, then $s^m$ contains $p$ since $m \geq 2\ell(p)$, which is a contradiction. 
    Moreover, we have $K \geq m$. Indeed, if $K<m$, then $s^m$ contains $r[k,K-1]r[0,k-1]$ for some $k$, so that $s^m$ contains $p_0$, which is a contradiction. 
    Thus, we have 
    \[ n+\ell(q)-\ell(p) \leq \frac LT \leq \frac{\ell(r)}{T} = \frac{2(n+\ell(q))}{T}.\]
    Hence, $T = 2$ and $L \geq 2(n+\ell(q) -\ell(p))$, so that we have $ps^A = s^B qt^C$ with 
    \[ A+B = m = n+\ell(q)-\ell(p),\ A+\ell(p) = B+\ell(q) + C =L/2. \]
    This implies that 
    \[ A \geq n+\ell(q)-2\ell(p),\ B \leq \ell(p),\ C \geq n-2\ell(p),\ A \geq B+C.\]
    In particular, we have $s=t$ and $ps^{A-C} = s^Bq$. This means that we can write 
    $p=s^Bw$ and $q=ws^{A-C}$ for some $w \in E^*_\cycle$. Since $q$ is primitive, $s^{B}ws^{A-B-C} \in [q]$ is also primitive, and is a primitive extension of $p$. This contradicts the assumption, so that the claim holds.  

    Let $r' = qt^nps^m \in [r]$, and we claim that $r'$ is a primitive extension of $q$. Suppose $r'[0,L-1]$ is not primitive for some $n+\ell(q) < L \leq 2(n+\ell(q))$, and put $r'[0,L-1] = r'[0,K-1]^T$ for some $K \leq L$ and $T \geq 2$. Since $q_0$ does not appear in $r'[\ell(q),n+\ell(q)-1]$, we have 
    \[ n \leq \frac LT \leq \frac{\ell(r)}{T} = \frac{2(n+\ell(q))}{T}.\]
    Hence, $T=2$ and $L \geq 2n$. Thus, we have $qt^A = t^B ps^C$ with 
    \[ A+B = n,\ A+\ell(q) = B+\ell(p) + C=L/2. \]
    This implies that 
    \[ A \geq n - \ell(q),\ B \leq \ell(q),\ C \geq n-\ell(p) -\ell(q),\ A \leq B+C.\]
    Hence, we have $s=t$.
    If $C \geq A$, then $q=s^Bps^{C-A}$, so that $ps^{B+C-A} \in [q]$. If $C \leq A$, then $qs^{A-C} = s^Bp$, so that 
    we can write $p=ws^{A-C}$ and $q=s^Bw$ for some $w \in E^*_\cycle$, which implies that $ps^{B+C-A}=ws^B \in [q]$. 
    Both cases contradict the assumption, so that $r'$ is a primitive extension of $q$. 
\end{proof}

\begin{proposition} \label{prop:transitive}
    Let $E$ be a strongly connected graph. 
    Let $p, q \in E^*$ be primitive cycles. Then, there exists $\alpha \in \Aut(\cG_E)$ such that $\alpha(p^\infty)=q^\infty$. 
    In particular, the action $\Out(\cG_E) \acts [E^*_\prim]$ is transitive. 
\end{proposition}
\begin{proof}
    If $q$ is a primitive extension of $p$, then there exists such a COE by using Lemma~\ref{lem:prim-transitivity} iteratively. 
    Hence, by Lemma~\ref{lem:4cycles} (1), we may assume that both $p$ and $q$ go through all edges. By replacing $p$ and $q$ with their cyclic permutations, we may assume in addition that $s(p)=s(q)$ because of Lemma~\ref{lem:existinner}. Now the claim follows from Lemma~\ref{lem:4cycles} (2), Lemma~\ref{lem:prim-transitivity}, and Lemma~\ref{lem:existinner}. 
\end{proof}

\section{Strong Proximality} \label{ssec:proximal}
In this section, we prove that the action $\Out(\cG_E) \acts \Pz M(X_E,\sigma_E)_+$ is strongly proximal. The proof of Theorem~\ref{thm:main} shall be completed in this section. 
Throughout this section, we fix a strongly connected graph $E$ which is not a subdivided circle. For a marker COE $\varphi$ and $v \in E^*$, $w \in E^*_\prim$ with $|v|\leq |w|$, we have 
\[ \frac{\langle v,\varphi(\eta_{[w]})\rangle}{\|\varphi(\eta_{[w]})\|} = \frac{\langle v, F_\varphi([w])\rangle}{|F([w])|}\]
by Proposition~\ref{prop:varphieta[p]}, Lemma~\ref{lem:rational-current}, and Lemma~\ref{lem:goodrep}. This fact is frequently used without any mention. 

\subsection{Partial equicontinuity lemmas of marker COEs} \label{ssec:technical}

Let $(e,f)\in E^2$ and $p \in E^*_\prim$. Assume that for all $n \geq 1$, $ep^nf$ does not contain an $ef$-segment, and $S(ep^nf,ep^nf) = \{|ep^nf|\}$ if $ e \neq f$ and $S(ep^nf,ep^nf) = \{|ep^nf|, 1\}$ if $ e = f$. 
Let $\varphi_n$ be the type I or type II marker COE for $\m=e,\m'=f$ and $D=\{o,p^n\}$. 
Let $F_n=F_{\varphi_n}$. 

In this subsection, we assume that $E^{[2]} \setminus ef$ is strongly connected. 
Let $[p] = \{ p(1),\dots,p(|p|)\}$ by Proposition~\ref{prop:general-cyclic}. For $k \geq 1$, let 
\[ [p]^k = \{ p(1)^k,\dots,p(|p|)^k \}.\]

\begin{lemma} \label{lem:easyoverlap}
    If $n, m \geq 1$ and $n \neq m$, then $S(ep^nf,ep^mf) =\emptyset$ if $e \neq f$ and $S(ep^nf,ep^mf) = \{1\}$ if $e=f$. 
\end{lemma}
\begin{proof}
    It suffices to show $S(ep^nf,ep^mf) \subseteq \{1\}$ and $S(ep^mf,ep^nf) \subseteq \{1\}$ by assuming $n < m$. 
    Suppose there exists $k \geq 2$ such that $(ep^mf)_{[0,k-1]}=(ep^nf)_{[n|p|-k+2,n|p|+1]}$. Since $k \leq n|p|+2 < |ep^mf|$, $(ep^mf)_{[0,k-1]}$ is a prefix of $ep^n$, so this contradicts the overlap condition for $ep^nf$. Suppose there exists $k \geq 2$ such that $(ep^nf)_{[0,k-1]}=(ep^mf)_{[m|p|-k+2,m|p|+1]}$. A similar argument shows that $k=|ep^nf|$. Then, the last edge of $p$ must be $e$, so $epf$ contains $ef$, which is a contradiction. 
\end{proof}

\begin{lemma} \label{lem:general-uniform}
    For any $\varepsilon>0$, $1 \geq \delta>0$, and $k\geq 1$, there exists $M>0$ such that 
    \[ \sum_{j=1}^{|p|} \frac{\langle p(j)^k, F_n([w])\rangle}{|F_n([w])|} \geq 1-\varepsilon \]
    for all $n \geq M$ and $[w] \in [E_\prim^*]$ with $\langle ef, [w] \rangle \geq \delta|w|$.
\end{lemma}
\begin{proof}
    Fix $n>k$ and 
    $[w] \in [E_\prim^*]$ with $\langle ef, [w] \rangle \geq \delta|w|$. Fix a good representative $w$ of $[w]$. 
    Let $K = \langle ef, [w]\rangle$, so that $K \geq \delta|w|$.  
    First, suppose either $\varphi$ is of type II, or $e=f$ and $w$ does not end with $e$. In this case, $K$ is equal to the number of $ef$-segments in $w$. 
    The map $F_n$ replaces each $ef$-segment with $ep^nf$.  
    Each new $ep^nf$-segment in $\bar{F}_n(w)$ arising this way contains at least $|p|(n-k)$ $[p]^k$-segments.
    This is because all segments in $p^n$ with length $|p|^k$ are elements in $[p]^k$. 
    Hence, 
    \begin{equation} \label{eqn:uniform1}
        \sum_{j=1}^{|p|}\langle p(j)^k,F_n([w]) \rangle \geq |p|(n-k)K.
    \end{equation}  
    All other segments in $w$ are either shrunk in length or left unchanged. Hence, 
     \begin{equation} \label{eqn:uniform2}
         |F_n([w])| \leq (n|p|+2)K + (|w|-2K) = n|p|K+|w|. 
     \end{equation}
    Next, suppose $e=f$ and $w$ begins and ends with $e$. In this case, the number of $ef$-segments in $w$ is $K-1$. 
    The map $F_n$ replaces each $e^2$-segment with $ep^ne$, and replaces the last $e$ with $ep^n$. Hence, the same argument shows that Inequality~\eqref{eqn:uniform1} holds in this case. Similarly, we have 
    \[  |F_n([w])| \leq (n|p|+2)(K-1) + (n|p|+1) + (|w|-(2K-1)) = n|p|K+|w|. \]
    Hence, Inequality~\eqref{eqn:uniform2} also holds in this case. 
     
    Thus, we have
    \[
    \sum_{j=1}^{|p|}\frac{\langle p(j)^k, F_n([w])\rangle}{|F_n([w])|} \geq 
    \frac{|p|(n-k)K}{n|p|K+|w|}
    = \frac{|p|(n-k)}{n|p|+|w|/K} \geq \frac{|p|(n-k)}{n|p|+\delta ^{-1}}. 
    \]
    Hence, 
    the claim holds by choosing $M>k$ such that
    \[ \frac{|p|(n-k)}{n|p|+\delta^{-1}} > 1-\varepsilon \]
    for $n \geq M$. 
\end{proof}
\begin{lemma} \label{lem:mu-uniform}
    For any $\varepsilon>0$, $1 \geq \delta>0$, and $k\geq 1$, there exists $M>0$ such that 
    \[ \sum_{j=1}^{|p|} \frac{\langle p(j)^k, \varphi_n(\mu)\rangle}{\|\varphi_n(\mu)\|} \geq 1-\varepsilon \]
    for all $n \geq M$ and $[\mu] \in \Pz M(X_E,\sigma_E)_+$ with $\langle ef, \mu \rangle \geq \delta\|\mu\|$.
\end{lemma}
\begin{proof}
    Let $\varepsilon>0$, $1 \geq \delta>0$, and $k\geq 1$. 
    Let $M$ be the constant in Lemma~\ref{lem:general-uniform} for $\varepsilon/2$ and $\delta/2$. 
    Let $n \geq M$ and $[\mu] \in \Pz M(X_E,\sigma_E)_+$ with $\langle ef, \mu \rangle \geq \delta\|\mu\|$.  
    By continuity of $\varphi_n$, Lemma~\ref{lem:seminorms}, and Proposition~\ref{prop:periodicmeasures}, there exists $[w] \in [E^*_\prim]$ such that 
    \[ 
    \left| \frac{\langle ef,\mu\rangle}{\|\mu\|} - \frac{\langle ef, [w] \rangle}{|w|}\right| < \frac{\delta}{2}\text{ and }
    \left| \sum_{j=1}^{|p|} \frac{\langle p(j)^k, \varphi_n(\mu)\rangle}{\|\mu\|} 
    - \sum_{j=1}^{|p|} \frac{\langle p(j)^k, F_n([w])\rangle}{|F_n([w])|}\right| < \frac{\varepsilon}{2}. \]
    Hence, the claim holds. 
\end{proof}

\begin{lemma}\label{lem:singleton}
    Let $[\mu] \in \Pz M(X_E,\sigma_E)_+$, and let $[\mu_n]\in \Pz M(X_E,\sigma_E)_+$ be a sequence.   
    \begin{enumerate}
        \item We have $[\mu] = [\eta_{[p]}]$ if and only if 
             \[ \sum_{j=1}^{|p|} \frac{\langle p(j)^k, \mu \rangle}{\| \mu\|}=1\]
             for all $k \geq 1$. 
        \item The sequence $[\mu_n]$ converges to $[\eta_{[p]}]$ if and only if 
        any $\varepsilon>0$ and $k\geq 1$, there exists $M>0$ such that 
        \[ \sum_{j=1}^{|p|} \frac{\langle p(j)^k, \mu_n\rangle}{\| \mu_n\|} \geq 1-\varepsilon \]
         for all $n \geq M$.
    \end{enumerate}
\end{lemma}
\begin{proof}
    (1): The only if direction is trivial. If $\|\mu\|=1$ and $\sum_{j=1}^{|p|} \langle p(j)^k, \mu \rangle = 1$ for all $k$, then $\mu$ is concentrated on 
    \[ \bigcap_{k=1}^\infty \bigcup_{j=1}^{|p|} \cZ(p(j)^k) = \{ p(1)^\infty, \dots, p(|p|)^\infty\}. \]
    Since $\mu$ is shift-invariant, we have 
    $\mu(\{p(1)^\infty\}) = \dots = \mu(\{p(|p|)^\infty\})$.

    (2): Let $[\mu]$ be an accummulation point of $[\mu_n]$. By assumption, we have 
    \[ \sum_{j=1}^{|p|} \frac{\langle p(j)^k, \mu\rangle}{\| \mu\|} =1 \]
    for all $k \geq 1$. Hence, $[\mu]=[\eta_{[p]}]$ by (1), so that $[\mu_n]$ converges to $[\eta_{[p]}]$. 
\end{proof}

The next lemma is the key observation for the analysis of measures $\mu$ such that $\langle ef, \mu\rangle = 0$. 
\begin{lemma} \label{lem:infsum}
    For every $[\mu] \in \Pz M(X_E,\sigma_E)_+$, we have 
    \[ \lim_{n \to \infty} \frac{n\langle ep^nf, \mu\rangle}{\| \mu \|} = 0.\]
\end{lemma}
\begin{proof}
    Fix $N \in \Zz_{>0}$. Then we have 
    \[ \sum_{n=1}^N \frac{n|p|\langle ep^nf, [w]\rangle}{|w|} \leq 1\]
    for all $[w] \in [E^*_\prim]$, because $ep^nf$-segments are disjoint or have overlap of length 1 by the overlap conditions and Lemma~\ref{lem:easyoverlap}, so each $ep^nf$-segment occupies a space of length at least $n|p|$ in $w$. Hence, by Proposition~\ref{prop:periodicmeasures} and Lemma~\ref{lem:seminorms}, we have 
    \[ \sum_{n=1}^N \frac{n|p|\langle ep^nf, \mu\rangle}{\| \mu\|} \leq 1\]
    for all $N$.
\end{proof}
\begin{lemma} \label{lem:mu-stable2}
    Let $n \in \Zz_{>0}$, $1/2>\varepsilon>0$, and let $[\mu] \in \Pz M(X_E,\sigma_E)_+$ with
    \[ \langle ef, \mu\rangle=0\ \mbox{ and }\
    \frac{\langle ep^nf, \mu\rangle}{\| \mu \|} < \frac{\varepsilon}{n|p|+2}.
    \]
    If $v \in E^*$ does not contain $ef$, then 
    \[ \frac{\langle v, \varphi_n(\mu)\rangle}{\|\varphi_n(\mu)\|} \leq \frac{1}{1-2\varepsilon}\frac{\langle v,\mu \rangle}{\|\mu\|}. \]
    If $v \in E^*$ contains $ef$, then 
    \[ \frac{\langle v, \varphi_n(\mu)\rangle}{\|\varphi_n(\mu)\|} \leq \frac{1}{1-2\varepsilon} \frac{\langle ep^nf, \mu\rangle}{\| \mu \|}. \]
\end{lemma}
\begin{proof}
    By continuity of $\varphi_n$, Lemma~\ref{lem:current-topology}, Lemma~\ref{lem:seminorms}, and Corollary~\ref{cor:rational-approx}, it suffices to show the claim for $[\mu]=[\eta_{[w]}]$ with $[w] \in [E^*_\prim]$ such that $\langle ef,[w] \rangle=0$. 
    In addition, we may assume that $|w|$ is large enough so that $\varepsilon|w| \geq n|p|+1$. 
    Fix a good representative $w$ of $[w]$. 
    The total length of all $ep^nf$-segments in $w$ is bounded by $\varepsilon |w|$ because of the overlap condition. In addition, 
    an $ep^n$-segment may appear as a suffix of $w$, but its length $n|p|+1$ is also bounded by $\eps|w|$ 
    (this is relevant only when $\varphi_n$ is of type I). Thus, the total length of the rest of $w$ is greater than $(1-2\varepsilon)|w|$. Hence, 
    \begin{equation} \label{eqn:length}
        (1-2\varepsilon)|w| < |F_n([w])| \leq |w|. 
    \end{equation}
    Since $F_n$ just decreases the number of $v$-segments in $w$ for all $v$ which do not contain $ef$, we have 
    \[ 
    \frac{\langle v, F_n([w])\rangle}{|F_n([w])|} \leq\frac{\langle v, [w]\rangle}{|w|} \frac{|w|}{F_n([w])} 
    \leq \frac{1}{1-2\varepsilon}\frac{\langle v, [w]\rangle}{|w|}.
    \]
    If $v$ contains $ef$, then an $ep^nf$-segment produces a new $ef$-segment, so it has the potential to produce a new $v$-segment. 
    In addition, if $\varphi$ is of type I and $w$ ends with $ep^n$, then it possibly produces a $v$-segment at the intermediate of $w^2$, which needs to be taken into account. 
    Since a new $v$-segment is only produced in this way, and 
    \[ \langle v,[w]\rangle \leq \langle ef, [w]\rangle = 0 \]
    by Lemma~\ref{lem:subword}, we have 
    \[ 
    \frac{\langle v, F_n([w])\rangle}{|F_n([w])|} 
    \leq \frac{\langle ep^nf, [w] \rangle}{|F_n([w])|}  
    \leq \frac{1}{1-2\varepsilon}\frac{\langle ep^nf, [w]\rangle}{|w|}.
    \]
\end{proof}
\begin{lemma} \label{lem:mu-stable}
    For any $\varepsilon>0$, there exists $\delta>0$ such that for all $n \in \Zz_{>0}$ and $v \in E^*$ such that $v$ does not contain $ef$, we have 
    \[ \left| \frac{\langle v, \varphi_n(\mu)\rangle}{\|\varphi_n(\mu)\|} - \frac{\langle v,\mu \rangle}{\|\mu\|}\right| < \varepsilon \]
    for $[\mu] \in \Pz M(X_E,\sigma_E)_+$ with
    \[ \langle ef, \mu\rangle=0\ \mbox{ and }\
    \frac{\langle ep^nf, \mu\rangle}{\| \mu \|} < \frac{\delta}{n|p|+2}.
    \]
\end{lemma}
\begin{proof}
    Let $n \in \Zz_{>0}$, let $v \in E^*$ such that $v$ does not contain $ef$, and let $0 < \varepsilon < 1/4$. Let $[\mu] \in \Pz M(X_E,\sigma_E)_+$ with
    \[ \langle ef, \mu\rangle=0\ \mbox{ and }\
    \frac{\langle ep^nf, \mu\rangle}{\| \mu \|} < \frac{\varepsilon}{n|p|+2}.
    \]
    It suffices to show 
    \[ \left| \frac{\langle v, \varphi_n(\mu)\rangle}{\|\varphi_n(\mu)\|} - \frac{\langle v,\mu \rangle}{\|\mu\|}\right| < 7\varepsilon. \]
    By continuity of $\varphi_n$, Lemma~\ref{lem:current-topology}, and Corollary~\ref{cor:rational-approx}, choose $[w] \in [E^*_\prim]$ such that 
    \begin{align*}
        &\langle ef, [w] \rangle = 0,\ \frac{\varepsilon}{n|p|+1}|w| \geq 1, \\
        &\left| \frac{\langle v,\mu \rangle}{\|\mu\|} - \frac{\langle v,[w] \rangle}{|w|}\right| < \varepsilon,\ \left| \frac{\langle v,\varphi_n(\mu) \rangle}{\|\varphi_n(\mu)\|} - \frac{\langle v,F_n([w]) \rangle}{|F_n([w])|}\right| < \varepsilon,\ 
    \frac{\langle ep^nf,[w] \rangle}{|w|} < \frac{\varepsilon}{n|p|+2}.
    \end{align*} 
    Fix a good representative $w$ of $[w]$. 
    All $ep^nf$-segments in $w$ shrink to $ef$ by $F_n$. 
    The number of $v$ in each $ep^nf$-segment is bounded by $|ep^nf| =n|p|+2$, so the total number of $v$ in all $ep^nf$-segments in $w$ is bounded by $\varepsilon |w|$.
    In addition, if $\varphi$ is of type I and $w$ ends with $ep^n$, then the last $ep^n$-segment also shrinks to $e$, and the number of $v$ we lose in this way is also bounded by $\varepsilon |w|$. 
    Hence, the ratio of the remaining $v$-segments is at least 
    \[ \frac{\langle v,[w] \rangle}{|w|} - 2\varepsilon > L-3\varepsilon,\]
    where $L = \langle v,\mu \rangle/\|\mu\|$. 
    Thus, 
    \[ (L+\varepsilon)|w| > \langle v,[w] \rangle \geq \langle v,F_n([w]) \rangle > (L-3\varepsilon)|w|.\]
    From the proof of Lemma~\ref{lem:mu-stable2}, the inequality \eqref{eqn:length} is valid here. 
    Hence, 
    \[ \frac{(L+\varepsilon)}{1-2\varepsilon} > \frac{\langle v,F_n([w]) \rangle}{|F_n([w])|} > L-3\varepsilon\]
    Thus, 
    \[
     \left| \frac{\langle v,\mu \rangle}{\|\mu\|} - \frac{\langle v,F_n[w] \rangle}{|F_n([w])|}\right| < 6\varepsilon.  
    \]
     Note that we have used 
    \[ \frac{L+\varepsilon}{1-2\varepsilon} - L \leq \frac{3\varepsilon}{1-2\varepsilon} < 6\varepsilon\]
    since $0 \leq L \leq 1$. 
    Therefore, we conclude that 
    \[ \left| \frac{\langle v, \varphi_n(\mu)\rangle}{\|\varphi_n(\mu)\|} - \frac{\langle v,\mu \rangle}{\|\mu\|}\right| < 6\varepsilon+\varepsilon. \]
\end{proof}

\subsection{The 2-edge-connected and non-$R_2$ case}

We assume that $E$ is 2-edge-connected and $E \not\cong R_2$. 
In this subsection, the proof for strong proximality shall be completed for such an $E$. 

\begin{lemma} \label{lem:allE^2}
    For every $(e,f) \in E^2$, 
    there exists a primitive cycle $p \in E^*_\prim$ at $r(e)=s(f)$ such that $epf$ contains an $e'f'$-segment for all $(e',f') \in E^2 \setminus \{(e,f)\}$, 
    $epf$ does not contain an $ef$-segment, 
    and 
    \[ S(ep^nf,ep^nf) = \begin{cases}
        \{|ep^nf|\} & \text{if } e \neq f, \\
        \{|ep^nf|, |e|\} & \text{if } e = f.
    \end{cases}\]
\end{lemma}
\begin{proof}     
    First, we consider the case $|E^0| \geq 2$. 
    Let $s^{-1}(s(f)) \setminus \{f\} = \{f_1,\dots,f_k\}$. Since $E$ is 2-edge-connected, we have $s^{-1}(s(f)) \setminus \{f\} \neq \emptyset$, and the set $s^{-1}(s(f)) \setminus \{f\}$ contains a non-loop edge, say $f_1$. For each $i=1,\dots,k$, choose a cycle $p_i$ at $s(f)$ starting with $f_i$ and ending with $e$, and $(p_i)_{[0,|p_i|-2]}$ does not contain $e$. Then, the path $ep_1\dots p_k$ contain a unique $ef_i$-segment for all $i=1,\dots,k$, and does not contain an $ef$-segment. 
    Let $S$ be the set of edges $g$ in $r^{-1}(s(e))$ such that the path $ep_1\dots p_k$ does not contain a $ge$-segment. If $S=\emptyset$, then let $q=p_1\dots p_k$. Otherwise, let $S=\{e_1,\dots,e_l\}$. 
    For each $j=1,\dots,l$, choose a path $q_j$ from $s(f)$ to $s(e)$ such that the first edge is $f_1$, the last edge is $e_j$, and $q_j$ does not contain $e$ (note that $e \not\in S$). This is possible because $E \setminus e$ is strongly connected. Let $q=p_1\dots p_k(q_1e)\dots(q_le)$. 
    By construction, either $ef_k$ or $e_le$ appears only once in $q$. 
    Since $E\setminus e$ is strongly connected, choose a cycle $r \in (E \setminus e)^*_\cycle$ with $|r|>|q|$ which starts from $f_1$ and contains an $e'f'$-segment for all $(e',f') \in (E\setminus e)^2$. Let $p=qr$. 
    
    If $s^{-1}(s(f)) \setminus \{f\} = \{f_1\}$, then $S\neq \emptyset$. This is because $s(e)$ receives at least 2 edges since $E$ is 2-edge-connected. Hence, either $ef_k$ or $e_le$ appears only once in $p$, so that $p$ is primitive. 
    Moreover, $p$ does not contain an $ef$-segment. We verify the overlap condition. Suppose we have $(ep^nf)_{[0,k-1]} = (ep^nf)_{[n|p|-k+2,n|p|+1]}$ for some $k >1$. By Proposition~\ref{prop:general-cyclic}, we have $k \leq |p|+1$ since $p$ does not end with $e$. 
    Note that $(ep^nf)_{[n|p|-k+2]} = e$. If $(ep^nf)_{[n|p|-k+2,n|p|+1]} = erf$, then $(ep^nf)_{[0,k-1]}$ contains at least 2 $e$-segments, although $r$ does not contain $e$, which is a contradiction. Hence, both $(ep^nf)_{[0,k-1]}$ and $(ep^nf)_{[n|p|-k+2,n|p|+1]}$ contain  $ef_k$ or $e_le$ which appears only once in $q$. Therefore, the place of $ef_k$ or $e_le$ must be the same, which forces $k > |p|+1$, so we get a contradiction. 

    Next, we consider the case $|E^0|=1$ and $e \neq f$. Let $E^1=\{e_0,\dots,e_N\}$ with $e_0=e,e_N=f$. By assumption, $N \geq 2$, so that $e_1 \neq e,f$. 
    Let 
    \[ q = e_0(e_0e_1)(e_0e_2) \dots (e_0e_{N-1})e_0e_1e_Ne_0. \]
    Then, all subwords of the form $eg$ or $ge$ for some $g \in E^1$ appear in $q$, and there is a unique $e_Ne_0$-segment in $q$. 
    Let $r$ be the product of all $e'f'$ for all $(e',f') \in (E \setminus e)^2$, starting from $e_1$ and ending with $f$. Choose $n$ so that $n|r|>|q|$. 
    Then, we can write $qr^n=epf$. By the same way, we can see that $p$ is primitive and satisfies the overlap conditions. This is because $e$ never appears after the unique $fe$-segment. The proof for the case $|E^0|=1$ and $e=f$ is similar. 
\end{proof}

We define the marker COEs as follows: Let $E^2=\{(e_1,f_1),\dots,(e_N,f_N)\}$. 
For each $i=1,\dots,N$, fix a primitive cycle $p_i$ as in Lemma~\ref{lem:allE^2} for $(e_i,f_i)$.
For $i=1,\dots,N$, let $\varphi_n^{(i)}$ be the type I marker COE for $\m=e_i$ and $D=\{o,p_i^n\}$ if $e_i=f_i$, and 
let $\varphi_n^{(i)}$ be the type II marker COE for $\m=e_i, \m'=f_i$ and $D=\{o, p_i^n\}$ if $e_i \neq f_i$. Then, the overlap conditions are satisfied by Lemma~\ref{lem:allE^2}. 
We can apply observations in Section~\ref{ssec:technical} by Proposition~\ref{prop:2edge}

\begin{lemma} \label{lem:converge-final}
    Let $E \neq R_2$. Then, 
    we have 
    \[ \lim_{n \to \infty} \varphi_n^{(1)}\dots\varphi_n^{(N)}([\mu]) = [\eta_{[p_1]}]\]
    for all $[\mu] \in \Pz M(X_E,\sigma_E)_+$. 
\end{lemma}
\begin{proof}
    Let $\varepsilon>0$, $K \in \Zz_{>0}$, and $[\mu] \in \Pz M(X_E,\sigma_E)_+$. 
    Let $1 \leq k \leq N$ be the largest index such that $\langle e_kf_k, \mu\rangle >0$. Choose $\alpha,\varepsilon_0>0$ such that 
    \[ \frac{\langle e_kf_k, \mu\rangle}{\|\mu\|} > \alpha +(N-k)\varepsilon_0.\]
    By Lemma~\ref{lem:mu-stable}, choose $2^{-N}>\delta>0$ such that 
    \[ \left| \frac{\langle e_kf_k, \varphi_n^{(i)}(\nu)\rangle}{\|\varphi_n^{(i)}(\nu)\|} - \frac{\langle e_kf_k,\nu \rangle}{\|\nu\|}\right| < \varepsilon_0 \]
    for all $i \in \{k+1,\dots,N\}$, $n \in \Zz_{>0}$, and $[\nu] \in \Pz M(X_E,\sigma_E)_+$ with
    \[ \langle e_{i}f_{i}, \nu\rangle=0\ \mbox{ and }\
    \frac{\langle e_{i}p_{i}^nf_{i}, \nu\rangle}{\| \nu \|} < \frac{2^{N-k-1}\delta}{n|p_{i}|+2}.
    \]
    By Lemma~\ref{lem:infsum}, choose $M_1>0$ such that 
    \[ \frac{\langle e_Np_N^nf_N, \mu\rangle}{\| \mu \|} < \frac{\delta}{nP+2} \]
    for all $n \geq M_1$, where $\displaystyle P= \max_{1 \leq i\leq N}  |p_i|$. Fix $n \geq M_1$. 
    By Lemma~\ref{lem:mu-stable2}, we have
    \begin{align*}
    &\frac{\langle e_{i}f_{i}, \varphi_n^{(N)}(\mu)\rangle}{\|\varphi_n^{(N)}(\mu)\|} \leq 2\frac{\langle e_{i}f_{i}, \mu\rangle}{\|\mu\|} =0 
    \text{ for } i=k+1,\dots,N-1, \text{ and}\\
    &\frac{\langle e_{N-1}p_{N-1}^nf_{N-1}, \varphi_n^{(N)}(\mu)\rangle}{\| \varphi_n^{(N)}(\mu) \|} 
    \leq 2\frac{\langle e_{N}p^n_Nf_{N}, \mu\rangle}{\|\mu\|}
    < \frac{2\delta}{nP+2} < \frac{2\delta}{n|p_{N-1}|+2} 
    \end{align*}
    since $p_{N-1}$ contains $e_Nf_N$, so that we obtain
    \[ \left| \frac{\langle e_kf_k, (\varphi_n^{(N-1)}\varphi_n^{(N)})(\mu)\rangle}{\|(\varphi_n^{(N-1)}\varphi_n^{(N)})(\mu)\|} - \frac{\langle e_kf_k,\varphi_n^{(N)}(\mu) \rangle}{\|\varphi_n^{(N)}(\mu)\|}\right| < \varepsilon_0. \] 
    Continuing this process, for $i=k+1,\dots,N$, we obtain 
    \begin{align*}
        &\langle e_{j}f_{j}, (\varphi_n^{(i)} \dots\varphi_n^{(N)})(\mu)\rangle =0
        \text{ for } j=k+1,\dots,i-1,\\
        &\frac{\langle e_{i-1}p_{i-1}^nf_{i-1}, (\varphi_n^{(i)} \dots\varphi_n^{(N)})(\mu)\rangle}
        {\| (\varphi_n^{(i)} \dots\varphi_n^{(N)})(\mu) \|} 
         < \frac{2^{N-i}\delta}{n|p_{i-1}|+2}, \text{ and} \\ 
        &\left| \frac{\langle e_kf_k, (\varphi_n^{(i)} \dots\varphi_n^{(N)})(\mu)\rangle}{\|(\varphi_n^{(i)} \dots\varphi_n^{(N)})(\mu)\|} - \frac{\langle e_kf_k, (\varphi_n^{(i+1)} \dots\varphi_n^{(N)})(\mu)\rangle}{\|(\varphi_n^{(i+1)} \dots\varphi_n^{(N)})(\mu)\|}\right| < \varepsilon_0.
    \end{align*}
    Therefore, 
    \[ \frac{\langle e_kf_k, (\varphi_n^{(k+1)} \dots\varphi_n^{(N)})(\mu)\rangle}
    {\|(\varphi_n^{(k+1)} \dots\varphi_n^{(N)})(\mu)\|} 
    > \frac{\langle e_kf_k,\mu \rangle}{\|\mu\|} - (N-k)\varepsilon_0 > \alpha \]
    holds for all $n \geq M_1$. 

    Let $\rho= \min \{ \alpha, (2P)^{-1}\}$. 
    Choose $\varepsilon_1>0$ such that $(1-\varepsilon_1)/P>\rho$. 
    By Lemma~\ref{lem:mu-uniform}, choose $M_2 \in \Zz_{>0}$ such that
    \[ \sum_{j=1}^{|p_i|} \frac{\langle p_i(j), \varphi_n^{(i)}(\nu)\rangle}{\|\varphi_n^{(i)}(\nu)\|} \geq 1-\varepsilon_1 \]
    for all $i\in\{2,\dots,k\}$, $n \geq M_2$ and $[\nu] \in \Pz M(X_E,\sigma_E)_+$ with $\langle e_if_i, [\nu] \rangle \geq \rho\|\nu\|$. 
    Since $p_i$ contains $e_{i-1}f_{i-1}$, we have by Lemma~\ref{lem:subword}
    \[ \frac{\langle e_{i-1}f_{i-1}, \varphi_n^{(i)}(\nu)\rangle}{\|\varphi_n^{(i)}(\nu)\|} \geq \frac{\langle p_i(j), \varphi_n^{(i)}(\nu)\rangle}{\|\varphi_n^{(i)}(\nu)\|}  \]
    for all $j$, so that 
    \[ \frac{\langle e_{i-1}f_{i-1}, \varphi_n^{(i)}(\nu)\rangle}{\|\varphi_n^{(i)}(\nu)\|} \geq \frac{1}{P}\sum_{j=1}^{|p_i|}\frac{\langle p_i(j), \varphi_n^{(i)}(\nu)\rangle}{\|\varphi_n^{(i)}(\nu)\|} \geq \frac{1-\varepsilon_1}{P} \geq \rho, \]
    for $i=2,\dots,k$ and $[\nu]$ with $\langle e_if_i, [\nu] \rangle \geq \rho\|\nu\|$. Hence, we can apply this argument recursively, 
    so that we obtain $\langle e_1f_1, \nu_n \rangle \geq \rho \|\nu_n\|$ for $n \geq \max\{M_1, M_2\}$, where 
    \[ [\nu_n] = (\varphi_n^{(2)}\dots \varphi_n^{(N)})([\mu]).\] 
    
    By Lemma~\ref{lem:mu-uniform}, choose $M>0$ such that $M > \max\{M_1, M_2\}$ and
    \[ \sum_{j=1}^{|p_1|} \frac{\langle p_1(j)^K, \varphi_n^{(1)}(\nu)\rangle}{\|\varphi_n^{(1)}(\nu)\|} \geq 1-\varepsilon \]
    for all $n \geq M$ and $\nu \in \Pz M(X_E,\sigma_E)_+$ with $\langle e_1f_1, [\nu] \rangle \geq \rho\|\nu\|$. 
    By Lemma~\ref{lem:singleton}, the sequence $\varphi_n^{(1)}([\nu_n])$ converges to $[\eta_{[p_1]}]$. 
\end{proof}

\begin{remark} \label{rem:discontinuous}
    From our approach, it is hard to show extreme proximality of $\Out(\cG_E) \acts \Pz M(X_E,\sigma_E)_+$, because the function 
    \[ G([\mu]):= \sum_{n=1}^{\infty} \frac{n|p|\langle ep^nf, \mu\rangle}{\| \mu\|}\]
    from Lemma~\ref{lem:infsum} is not continuous. For example, if $e=f$ and $e$ is a loop, then $ep^ne$ is a primitive cycle, and $[\mu_n] := [\eta_{[ep^ne]}]$ converges to $[\mu]:=[\eta_{[p]}]$. We have 
    \[ G([\mu_n]) \geq \frac{n|p|\langle ep^ne, \mu\rangle}{\| \mu\|} = \frac{n|p|}{n+2}, \]
    but $G([\mu])=0$. 
\end{remark}

\subsection{The general case}

Until here, we have proved strong proximality for 2-edge-connected graphs except for $R_2$. The general case is reduced to the 2-edge-connected case by using classification results of Cuntz--Krieger groupoids. 

\begin{proposition} \label{prop:EKTW}
    For any strongly connected graph $E$ such that $E \not\cong S^1_n$, there exists a 2-edge-connected graph $F$ such that $\cG_E\cong\cG_F$. 
\end{proposition}
\begin{proof}
    Let $A_E$ be the adjacency matrix of $E$. Let 
    \[ \H_0(\cG_E) = \Zz^n \oplus \bigoplus_{i=1}^k \Zz/m_k\Zz,\]
    where $m_i|m_{i+1}$. 
    Then, by the proof of \cite[Proposition~3.6]{EKTW16}, there exist $b_1,\dots,b_{n+k} \geq 1$ such that the matrix
    \[ B= \begin{bmatrix}
        1 & 1 & 1 &  \dots & 1 \\
        b_1 & b_1+m_1 & b_1 & \dots &b_1 \\
        \vdots & \vdots & \ddots & \vdots & \vdots \\
        b_{k+n} & b_{k+n} & b_{k+n} & \dots & b_{k+n}
    \end{bmatrix}
    \]
    satisfies $\Coker B \cong \H_0(\cG_E)$, $\Ker B \cong \H_1(\cG_E)$, and $[{\bf 1}] = [1_{E}]_0$ modulo $\Im B$, where ${\bf 1}$ denotes the vector whose entries are all $1$. Let
    \[
    C= 
        \newsavebox{\leftblock}
        \sbox{\leftblock}{$\begin{matrix}0 & 1 \\[2pt] 1 & 0\end{matrix}$}
        \left[
        \begin{array}{@{}c|c@{}}
        \usebox{\leftblock} &
        \makebox[\wd\leftblock][c]{\raisebox{0pt}[\ht\leftblock][\dp\leftblock]{\Large 0}}
        \\ \hline
        \makebox[\wd\leftblock][c]{\raisebox{0pt}[\ht\leftblock][\dp\leftblock]{\Large 0}} &
        \makebox[\wd\leftblock][c]{\raisebox{0pt}[\ht\leftblock][\dp\leftblock]{\Large I}}
        \end{array}
        \right],
    \]
    where the right-bottom $I$ is the identity matrix of size $k+n-1$. 
    Let $A_F = B^t+I$ if $\sgn(\det(B)) =\sgn(\det(A_E-I))$, and let $A_F=(BC)^t + I$ if $\sgn(\det(B)) \neq \sgn(\det(A_E-I))$. Let $F$ be the graph whose adjacency matrix is $A_F$. Then, by \cite[Theorem~9.6.1]{MatBook}, we have $\cG_E \cong \cG_F$. 
    Note that since $\Im B = \Im BC$, multiplying $C$ to the right does not change $\H_0(\cG_E) = \Coker (I-A_F^t)$ nor the position of the unit. 

    If $k+n \geq 2$, then $F$ is 2-edge-connected since all entries of $A_F$ are positive and $|F^{0}| \geq 3$. Suppose $k+n=1$. Then, $A_F$ is of the form 
    \[ B=\begin{bmatrix}
        a & b \\ c & d
    \end{bmatrix}
    \]
    with $a,d \geq 2$ and $b,c \geq 1$. We can directly see that the 2nd power of $F$ is 2-edge-connected. Finally, the case $k=n=0$ is already covered by letting $n=0,k=1$, and $m_k=1$ (note that the construction of \cite[Proposition~3.6]{EKTW16} works even if $m_1=1$). 
\end{proof}

\begin{proposition} \label{prop:strongly-proximal}
    For any strongly connected graph $E$ such that $E \not\cong S^1_n$, there exists a sequence $\varphi_n \in \COE(X_E,\sigma_E)$ and $[p] \in [E^*_{\prim}]$ such that $\varphi_n([\mu])$ converges to $[\eta_{[p]}]$ for all $[\mu] \in \Pz M(X_E,\sigma_E)_+$. 
\end{proposition}
\begin{proof}
    By Proposition~\ref{prop:EKTW}, there exists a 2-edge-connected graph $F$ such that $\cG_E \cong \cG_F$. By \cite[Theorem~9.6.1]{MatBook}, there exists a COE $\psi \colon (X_F,\sigma_F) \to (X_E,\sigma_E)$. By Lemma~\ref{lem:converge-final}, there exists a sequence $\varphi_n \in \COE(X_F, \sigma_F)$ and $[p] \in [F^*_{\prim}]$ such that $\varphi_n([\mu])$ converges to $[\eta_{[p]}]$ for all $[\mu] \in \Pz M(X_F,\sigma_F)_+$.  
    Let $\psi[p]=[q]$ for $[q] \in [E^*_\prim]$. 
    Then, $\psi\eta_{[p]}=\eta_{[q]}$ by Proposition~\ref{prop:varphieta[p]}. Since $\psi \colon \Pz M(X_F,\sigma_F)_+ \to \Pz M(X_E,\sigma_E)_+$ is a homeomorphism, we see that $\psi \circ \varphi_n \circ \psi^{-1}([\mu])$ converges to $[\eta_{[q]}]$ for all $[\mu] \in \Pz M(X_E,\sigma_E)_+$. 
\end{proof}

\begin{proof}[Proof of Theorem~\ref{thm:main}]
Strong proximality and minimality of the action $\Out(\cG_E) \acts \Pz M(X_E,\sigma_E)_+$ follow from Proposition~\ref{prop:strongly-proximal} and Proposition~\ref{prop:transitive} by a standard argument using the dominated convergence theorem.  

Let $\{[p_1], [p_2],\dots\}$ be an enumeration of $[E^*_\prim]$ obtained by arranging the elements of $[E^*_\prim]$ in increasing order of word length. Let 
\[ \mu:= \sum_{i=1}^\infty\frac{1}{2^i}\eta_{[p_i]}\in M(X_E,\sigma_E)_+.\]
Note that the infinite sum converges in norm because the growth of $\|\eta_{[p_i]}\|=|p_i|$ is linear. 
Let $\alpha \in \Aut(\G_E)$, and suppose $\alpha[\mu]=\mu$. Then, we have $\alpha(\mu)=r\mu$ for some $r>0$, so that 
\[ \alpha(\mu)=\sum_{i=1}^\infty\frac{1}{2^i}\eta_{\alpha[p_i]} = \sum_{i=1}^\infty\frac{r}{2^i}\eta_{[p_i]}\]
by Proposition~\ref{prop:varphieta[p]}. 
Define $F \colon E^*_\prim \to \Rz$ by $F(p) = \alpha(\mu)(\{p^\infty\})$, and let 
$A \subseteq \Rz$ be the range of $F$. Enumerate $A=\{a_1,a_2,\dots\}$ with $a_1 > a_2 > \dots$. Then, we see that $\alpha[p_i] = F^{-1}(a_i) = [p_i]$ in $E^*_\prim$. Hence, $\alpha[p_i]=[p_i]$ for all $i$, so that $\alpha$ is inner by Corollary~\ref{cor:faithful-on-ep}. 
This shows topological freeness. 
Therefore, the action of $\Out(\F(\cG_E)) \cong \Out(\cG_E)$ on $\Pz M(X_E,\sigma_E)_+$ is a topologically free boundary action by Proposition~\ref{prop:aut}. 
Now C*-simplicity follows from Theorem~\ref{thm:KK}. 

Note that periodic measures are ergodic, so that they are extremal points of the simplex $P(X_E,\sigma_E)$. Hence, Proposition~\ref{prop:periodicmeasures} implies $\Pz M(X_E,\sigma_E)_+$ is homeomorphic to the Poulsen simplex. Now $\Pz M(X_E,\sigma_E)_+$ is homeomorphic to the Hilbert cube by \cite[Theorem~3.1]{LOS78}.
\end{proof}

\subsection{Some remarks on strong proximality} \label{ssec:remark}
In the proof of strong proximality in Proposition~\ref{prop:strongly-proximal}, we used the classification result by Matsumoto \cite[Theorem~9.6.1]{MatBook}, which sits within the long historical development of the classification of Kirchberg algebras. For readers who may not be familiar with operator algebras, we collect outlines of direct proofs of strong proximality of the actions of $\Out(V_2)$ and $\Out(V_{n,r})$ without using the classification result. 

Let us consider the case $E=R_2$, and let $E^1=\{ e,f\}$. 
Let $\varphi_n^{(1)}$, $\varphi_n^{(2)}$ be the type I marker COEs for $\m = e$, $D=\{o,f^n\}$ and for $\m=f$, $D=\{o,e^n\}$, respectively. In addition, let $\varphi^{(3)}_n$ be the type I marker COE for $\m=ef$, $D=\{o,f^n\}$. 
Note that the overlap conditions for $\varphi_n^{(3)}$ are satisfied because the first $eff$ is the unique $eff$-segment in $eff^nef$. Then, the same proof works. The key point is to apply Lemma~\ref{lem:embedding} for $N=4$ and $w=efef$ since $E^{[4]} \setminus efef$ is strongly connected. Hence, the action of $\Out(\cG_E)=\Out(V_2)$ on $\Pz M(X_E,\sigma_E)_+$ is strongly proximal. 

Now let $E$ be the graph as in Section~\ref{ssec:HT}, and let $F=E_0 \cong R_n$. Then, we see that $\cG_E \cong \cG_F \times \cR_r$, where 
$\cR_r = \{0,\dots,r-1\}^2$ 
is the full equivalence relation on the $r$-point set (note that the multiplication is defined by $(i,j)(j,k)=(i,k)$). Indeed, for $k=0,\dots,r-1$, let 
\[ U_k = \{ (\sigma_E^k(x), -k, x) \in \cG_E \mid x \in X_0\}. \]
Define $\alpha \colon \cG_F \times \cR_r \to \cG_E$ by 
$\alpha(g, (i,j)) = U_igU_j^{-1}$.
Then, we can see that $\alpha$ is an isomorphism. Then, we have an embedding $\Out(\cG_F) \to \Out(\cG_E)$ by $\beta \mapsto \beta \times \id$. We see that this embedding makes the isomorphism $\H^1(\cG_E) \cong \H^1(\cG_F)$ in Proposition~\ref{prop:H1isom} (see below) equivariant, so that it makes the isomorphism $\Pz M(X_F,\sigma_F)_+ \cong \Pz M(X_E,\sigma_E)_+$ equivariant. Therefore, we can find a sequence $g_n \in \Out(\cG_E) = \Out(V_{n,r})$ as in Proposition~\ref{prop:strongly-proximal} from the image of $\Out(\cG_F)$. 

The following well-known fact can be easily deduced from the proof of \cite[Theorem~3.6]{Mat12}.
\begin{proposition} \label{prop:H1isom}
    If $U\subseteq \cG^{(0)}$ is a full clopen set, then the restriction map $\Hom(\cG,\Zz)\to\Hom(\cG_U^U,\Zz)$ induces an isomorphism $\H^1(\cG)\cong\H^1(\cG_U^U)$.
\end{proposition}
\begin{proof}
    Let $\theta\colon\cG^{(0)}\to\cG$ be a continuous map such that $\r(\theta(x))=x$ and $\s(\theta(x))\in U$ for all $x\in\cG^{(0)}$. Such a map exists by the proof of \cite[Theorem~3.6(2)]{Mat12}. Let $\rho\colon \cG\to\cG_U^U$ be the homomorphism given by $\rho(g)\coloneqq \theta(\r(g))^{-1}g\theta(\s(g))$ (the map $\rho$ is from the proof of \cite[Theorem~3.6]{Mat12}). 

    It is not difficult to check that the homomorphism $\Hom(\cG_U^U,\Zz)\to\Hom(\cG,\Zz)$ given by $\omega\mapsto \omega\circ\rho$ maps coboundaries to coboundaries and thus induces a homomorphism $\H^1(\cG_U^U)\to\H^1(\cG)$. We show that it is the inverse of the map induced by restriction. For $\omega\in\Hom(\cG,\Zz)$, we have $\omega\vert_{\cG_U^U}\circ\rho=\omega+\omega\circ\theta\circ\s-\omega\circ\theta\circ\r$, so that $[\omega\vert_{\cG_U^U}]=[\omega]$ in $\H^1(\cG)$. Similarly, $[(\omega\circ\rho)\vert_{\cG_U^U}]=\omega$ for all $\omega\in\Hom(\cG_U^U,\Zz)$.
\end{proof}

\section{Consequences}

\subsection{Brin--Thompson groups}
For $d\geq 1$, let $dV=dV_2$ denote the associated Brin--Thompson group from \cite{Brin04}.

\begin{lemma} \label{lem:wreath}
    Let $H$ be a non-trivial C*-simple group, and let $G=H\wr \fS_d$. Then, $G$ is C*-simple.
\end{lemma}
\begin{proof}
    By \cite[Proposition~19(i)]{delaHarpe07}, the group $H^d$ is C*-simple. Thus, by \cite[Proposition~19(v)]{delaHarpe07}, it suffices to show that the centraliser of $H^d$ in $H\wr \fS_d = H^d \rtimes \fS_d$ is trivial (cf. \cite[Theorem~1.4]{BKKO17}). 
    Suppose $(\alpha_1,...,\alpha_d)\tau\in H\wr \fS_d$ centralises $H^d$. Then, we have 
    \[ (\beta_1,...,\beta_d)(\alpha_1,...,\alpha_d)\tau=(\alpha_1,...,\alpha_d)\tau(\beta_1,...,\beta_d)=(\alpha_1,...,\alpha_d)(\beta_{\tau(1)},...,\beta_{\tau(d)})\tau, \]
    so that $\beta_i\alpha_i=\alpha_i\beta_{\tau(i)}$ for all $(\beta_1,\dots,\beta_d) \in H^d$. Since $H$ has trivial centre, we have $\alpha_i=1$ for all $i$ by letting $\beta_1=\dots=\beta_d=\beta \in H$. Hence, we have $\beta_{i}=\beta_{\tau(i)}$ for all $(\beta_1,\dots,\beta_d) \in H^d$. Since $|H| = \infty$, we conclude that $\tau=1$. 
\end{proof}
\begin{corollary}
    For all $d\geq 1$, the outer automorphism group of the Brin--Thompson group $\Out(dV)$ is C*-simple.
\end{corollary}
\begin{proof}
    This follows from Theorem~\ref{thm:main}, Lemma~\ref{lem:wreath}, and \cite[Theorem~1.1]{Elliott23}. 
\end{proof}

\subsection{Amenable radicals of commutator subgroups of TFGs}

The following proposition is proved in the same way as Proposition~\ref{prop:aut} using \cite[Proposition~3.6]{Mat15}.
\begin{proposition}
\label{prop:Aut(F)=Aut(D)}
 Assume $\cG^{(0)}$ is a Cantor set. Then, the restriction map $\res\colon\Aut(\F(\cG))\to \Aut(\D(\cG))$ is a group isomorphism.
\end{proposition}

\begin{corollary}
\label{cor:ses}
 Assume that $\cG$ is purely infinite and that $\cG^{(0)}$ is a Cantor set. Then, there is a canonical short exact sequence
\[
1\to \F(\cG)/\D(\cG)\to\Out(\D(\cG))\to\Out(\F(\cG))\to 1.
\]
\end{corollary}
\begin{proof}
The composition $\Aut(\D(\cG))\cong\Aut(\F(\cG))\to\Out(\F(\cG))$ has kernel equal to $\Gamma/\Inn(\D(\cG))$, where $\Gamma$ is the isomorphic copy of $\Inn(\F(\cG))$ in $\Aut(\D(\cG))$ under the restriction map. Since $\D(\cG)$ and $\F(\cG)$ have trivial centres by \cite[Theorem 4.16]{Mat15}, the result follows.
\end{proof}

We now turn to the proof of Corollary~\ref{cor}.
\begin{proof}[Proof of Corollary~\ref{cor}]
Since $\F(\cG_E)/\D(\cG_E)$ is amenable, and $\Out(\F(\cG_E))$ is C*-simple by Theorem~\ref{thm:main}, it follows from by \cite[Lemma 6.1]{Mat15} and Corollary~\ref{cor:ses} that the amenable radical of $\Out(\D(\cG_E))$ is isomorphic to $\F(\cG_E)/\D(\cG_E)$ and that $\Out(\D(\cG_E))$ is C*-simple if and only if $\F(\cG_E)/\D(\cG_E)$ vanishes.

By \cite[Corollary 6.24(1)]{Mat15}, we have $\F(\cG_E)/\D(\cG_E)\cong (\H_0(\cG_E)\otimes\Zz/2\Zz)\oplus\H_1(\cG_E)$. Finally, by \cite[Theorem 4.14]{Mat12}, we have $\H_0(\cG_E)\cong \Coker(I-A_E^t)$ and $\H_1(\cG_E)\cong \Ker(I-A_E^t)$.
\end{proof}

\end{document}